\documentclass[11pt]{amsart}
\usepackage{amsthm}
\usepackage{amsmath,amstext,amsthm,amsfonts,amssymb,amsthm}
\usepackage[backref]{hyperref}
\usepackage{multicol}
\usepackage{latexsym}
\usepackage{color}
\usepackage{graphicx}
\usepackage{epsf}
\usepackage{epsfig}
\usepackage{epic}
\usepackage{graphics}
\usepackage{verbatim}
\usepackage{color}
\usepackage{hyperref}
\newtheorem{theorem}{Theorem}[section]
\newtheorem{lemma}[theorem]{Lemma}
\theoremstyle{definition}
\newtheorem{definition}[theorem]{Definition}
\newtheorem{proposition}[theorem]{Proposition}
\newtheorem{example}[theorem]{Example}

\theoremstyle{remark}
\newtheorem{remark}[theorem]{Remark}
\numberwithin{equation}{section}

\newcommand{\R}{\mathbb{R}}

\begin{document}
\title[Lagrange Multiplier Sufficient Conditions]{Lagrange Multiplier Local Necessary and Global Sufficiency Criteria for Some Non-Convex Programming Problems}
\author[Muraleetharan, B. {\em et al.}]{B. Muraleetharan$^{\dagger}$, S. Selvarajan$^{\dagger}$, S. Srisatkunarajah$^{\dagger}$ and\\ K. Thirulogasanthar$^{\ddagger}$}
\address{$^{\dagger}$ Department of mathematics and Statistics, University of Jaffna, Thirunelveli, Sri Lanka.}
\address{$^{\ddagger}$ Department of Computer Science and Software Engineering, Concordia University, 1455 De Maisonneuve Blvd. West, Montreal, Quebec, H3G 1M8, Canada.}
\email{bbmuraleetharan@jfn.ac.lk, ssrajan90@gmail.com, srisatku@yahoo.com, santhar@gmail.com}
\thanks{K. Thirulogasanthar would like to thank the FQRNT, Fonds de la Recherche  Nature et  Technologies (Quebec, Canada) for partial financial support under the grant number 2017-CO-201915. Part of this work was done while he was visiting the University of Jaffna to which he expresses his thanks for the hospitality.}
\date{\today}
\begin{abstract}
In this paper we consider three minimization problems, namely quadratic, $\rho$-convex and quadratic fractional programing problems. The quadratic problem is considered with quadratic inequality constraints with bounded continuous and discrete mixed variables. The $\rho$-convex problem is considered with $\rho$-convex inequality constraints in mixed variables. The quadratic fractional problem is studied with quadratic fractional constraints in mixed variables. For all three problems we reformulate the problem as a mathematical programming problem and apply standard Karush Kuhn Tucker necessary conditions.  Then, for each problem, we provide local necessary optimality condition. Further, for each problem a Lagrangian multiplier sufficient optimality condition is provided to identify global minimizer among the local minimizers. For the quadratic problem underestimation of a Lagrangian was employed to obtain the desired sufficient conditions. For the $\rho$-convex problem we obtain two sufficient optimality conditions to distinguish a global minimizer among the local minimizers, one with an underestimation of a Lagrangian and the other with a different technique.  A global sufficient optimality condition for the quadratic fractional problem is obtained by reformulating the problem as a quadratic problem and then utilizing the results of the quadratic problem. Examples are provided to illustrate the significance of the results obtained.
\end{abstract}
\maketitle
\pagestyle{myheadings}
\section{Introduction}\label{sec_intro}\noindent
In classical Calculus, method of Lagrange multiplier provides first order necessary condition for optimization problems with equality constraints.  Celebrated Karush-Kuhn-Tucker (KKT) conditions, published in 1951, generalize the Lagrange multiplier approach to Mathematical Programming problems with both equality and inequality constraints \cite{KKT}. Solutions of constrained convex optimization problems have been long studied and to such a problem a local extrema is a global one \cite{Par,pr}. However, the non-convex optimization problems pose NP-hard challenges.  Since local necessary optimality conditions play an important role in identifying local minimizers, recently, attempts have been made to formulate local optimality conditions, analogous to KKT conditions, to non-linear programming problem with bounded variables/box constraints. A common optimization problem in many real-world applications is to identify and locate a global minimizer of functions of several variables with bounds on the variables \cite{fp}.  Non-convex quadratic optimization problems have numerous applications such as electronic circuit design, computational chemistry, combinatorial optimization and many more. See, for example, \cite{fp} and the many references therein.\\

In this work we consider the mathematical programming model problem:
\begin{eqnarray*}
\mbox{(MP)~~}~~\min_{x\in\mathbb{R}^{n}} f_{0}(x)&~&~\\
\mbox{s.t.~~}~~ f_{j}(x) &\leq0,& ~~\forall~j\in\{1,2,3,\cdots,m\}\\
x_{i} \in [u_{i},v_{i}],&i\in I & \mbox{-\,Continuous variable}\\
x_{i} \in \{u_{i},v_{i}\},&i\in J ;& \mbox{-\,Discrete variable},
 \end{eqnarray*}
where $I\cap J=\phi$, $I\cup J=\{1,2,3,\cdots,n\}$.
We study three particular cases: $f_{j}:j=0,1,2,\dots,m$ are quadratic functions, are $\rho-$convex functions and  are quadratic fractional functions. \\

 Characterizing global solutions to constrained non-convex problems that exhibit multiple local extrema has been limited to a few classes of problems \cite{Beck, Hor, SriJaya1, Jey, Jey2, Pi},  to name a few. Sufficient global optimality condition has been studied for a quadratic optimization problem with binary constraints in \cite{Beck, Pi}. For a quadratic optimal function global optimality conditions were developed in \cite{Jey2, Pe}. In \cite{Li} the authors considered non-convex minimization of a twice differentiable function with quadratic inequality constraints and obtained necessary condition for a global minimizer. They have also studied the same optimization problem with only the box constraints and obtained necessary and sufficient conditions for a global minimizer. In addition they have obtained necessary and sufficient conditions for a quadratic objective function with only the box constraints. 
However, in this article, in some sense as a unification of box and bivalent constraints considered to a quadratic minimization problem, we provide local necessary and sufficient condition for a quadratic minimization problem with quadratic, box and discrete variable constraints altogether, that is, we consider the problem MP when  $f_{j}:j=0,1,2,\dots,m$ are quadratic functions (see P1) . We also provide sufficient condition for the same problem to distinguish the global minimizer among the local minimizers. There is also another reason for studying this problem at first, which is the results obtained in this quadratic problem put a platform in obtaining the sufficient condition for fractional optimization problem with fractional constraints.\\

In \cite{Wu} sufficient global optimality condition for weakly convex minimization problems with weakly convex inequality and equality constraints, using abstract convex analysis theory, has been obtained. However, in obtaining such conditions they have used the so-called $(L, X)$-subdifferentials. In this note, we consider the problem MP when  $f_{j}:j=0,1,2,\dots,m$ are $\rho$-convex functions (see P2) and we obtain verifiable sufficiency condition for locating the global minimizer among the local minimizers. Further, the conditions we derive are different from the one given in \cite{Wu} and easy to verify. In addition, in our optimization problem we have mixed variables, namely continuous and discrete and obtain a sufficient condition in a verifiable semi-definite matrix form, which itself makes our problem different from the one considered in \cite{Wu}.\\

Constrained fractional programing problems have a wide range of applications such as signal processing, communications, location theory etc. \cite{Lo, Oh}. In \cite{Sa} the authors have studied a quadratic fractional optimization problem with strictly convex quadratic constraints using Newton’s algorithm. In \cite{Za} a quadratic fractional optimization problem has been studied with two quadratic convex constraints using the classical Dinkelbach approach with no global convergence guarantee. In  \cite{Ab}, using the conditional gradient method the ratio of two convex functions over a closed and convex set has been studied numerically. In this article we consider the problem MP when  $f_{j}:j=0,1,2,\dots,m$ are quadratic fractional functions (see P3). For this problem, we first provide local necessary conditions and then obtain a verifiable sufficient condition for identifying global minimizer among the local ones. In this regard, our problem generalizes the cases considered in \cite{Sa, Za} and \cite{Ab} in a considerable way.

\section{Preliminaries}
\subsection{Notations:}
We shall begin with the notations and definitions. For the problem (MP), let
$$\Delta=\{x=(x_{i})\in\mathbb{R}^{n}~|~x_{i}\in[u_{i},v_{i}], ~i\in I\},$$
$$D=\{x\in\mathbb{R}^{n}~|~x_{i}\in[u_{i},v_{i}], i\in I~~\mbox{~and~}~~x_{i}\in\{u_{i},v_{i}\}, i\in J\},$$
where $I\cap J=\phi$, $I\cup J=\{1,2,3,\cdots,n\}$,
and 
$$\tilde{D}=\{x\in D~|~f_{j}(x)\leq 0,~~\forall~j\in\{1,2,3,\cdots,m\}~\}:~\mbox{~~feasible set}.$$
Now for $\lambda=(\lambda_{j})\in\mathbb{R}^{m}_{+}$, define the Lagrangian $L(\cdot,\lambda)$ of (MP) by
 \begin{equation}\label{lag_MP}
L(x,\lambda):=f_{0}(x)+\sum_{j=1}^{m}\lambda_{j}f_{j}(x);~~x\in\mathbb{R}^{n}.
 \end{equation}
By taking $\lambda_{0}=1$ the Lagrangian of (MP) can be written as $$L(x,\lambda)=\sum_{j=0}^{m}\lambda_{j}f_{j}(x).$$
For $\overline{x}\in D$ and $i\in \{1,2,3,\cdots,n\}$, define
\begin{equation}\label{eq4}
\chi_{i}(\overline{x}):=\left\{
\begin{array}{lll}
-1 & \mathrm{if}\ \overline{x}_i=u_i \\
~~~~\,\,\,1 & \mathrm{if}\ \overline{x}_i=v_i\\
\nabla L(\overline{x},\lambda)_{i} & \mathrm{if}\ \overline{x}_i\in (u_i ,v_i ).
\end{array}%
\right. 
\end{equation}
For a symmetric matrix $A$, $A\succeq0$ means that $A$ is positive semi-definite, and $A\succ0$ means that $A$ positive definite.
\begin{definition}\cite{JeySriHuy}
	A quadratic function $\ell:\mathbb{R}^{n}\longrightarrow\mathbb{R}$ is said to be a minimizing quadratic \textit{under-estimator} of a function $f:\mathbb{R}^{n}\longrightarrow\mathbb{R}$
	at $\overline{x}$ over a set $\Omega$, if for each $x\in\Omega$ $$f(x)-f(\overline{x})\geq \ell(x)-\ell(\overline{x})\geq0.$$
\end{definition}
Let \begin{equation}\label{mu}
\mu_j=\inf_{\Vert x \Vert =1} \frac{-1}{2}x^TA_jx, 
\end{equation}
be the first eigenvalue of $\displaystyle{\frac{-A_j}{2}},$ where $A_{j}=(a_{st}^{(j)})$ is an $n\times n$ symmetric matrix, for all $j\in\{0,1,2,3,\cdots,m\}$.
\begin{definition}\cite{J-P} \ A function $f: C \subseteq \R^n \rightarrow \R$ from a  convex subset $C$ of $\R^n$ into $\R$ is said to be \textbf{convex}, if $\forall x_1,\;x_2\in C$ and $\forall \lambda \in [0,1],$ 
	\begin{equation} \label{w1}
	f(x_1+(1-\lambda)x_2 )\le \lambda f(x_1) + (1-\lambda) f(x_2).
	\end{equation}
\end{definition} 
\begin{definition} \ A function $f: C \subseteq \R^n \rightarrow \R$ from a  convex subset $C$ of $\R^n$ into $\R$ is said to be \textbf{$\rho$- convex}, if there exists $\rho \in \R$ such that $\forall x_1,\;x_2\in C$ and $\forall \lambda \in [0,1],$ 
	\begin{equation} \label{w1}
	f(x_1+(1-\lambda)x_2 )\le \lambda f(x_1) + (1-\lambda) f(x_2) -\rho \lambda (1-\lambda) \Vert x_1-x_2 \Vert.
	\end{equation}
\end{definition} 
\begin{remark}
	Condition $(\ref{w1})$ is equivalent to say that $f(x)-\rho \Vert x \Vert^2$ is a convex function over $C$.
\end{remark}
The following proposition gives a fundamental result.
\begin{proposition}\label{p1}
	For each $j\in\{0,1,2,3,\cdots,m\}$, we have
	$$\mu_j=\inf_{\Vert x \Vert =1} \frac{-1}{2}x^TA_jx=\inf_{x\neq0}\frac{-1}{~2}\frac{x^TA_jx}{\|x\|^{2}}.$$
\end{proposition}
\begin{proof}
	Let $j\in\{0,1,2,3,\cdots,m\}$ and $x\in\tilde{D}$ with $x\neq0$, then $\displaystyle\left\|\frac{x}{\|x\|}\right\|=1$. Therefore
	$$\inf_{\Vert x \Vert =1} \frac{-1}{~2}x^TA_jx\leq\frac{-1}{~2}\left(\frac{x}{\|x\|}\right)^TA_j\left(\frac{x}{\|x\|}\right)=\frac{-1}{~2}\,\frac{x^TA_jx}{\|x\|^{~2}},~~\forall\,x\in\tilde{D}\smallsetminus\{0\}.$$
	Thus
	\begin{equation}\label{ine1}
		\inf_{\Vert x \Vert =1} \frac{-1}{2}x^TA_jx\leq\inf_{x\neq0} \frac{-1}{~2}\frac{x^TA_jx}{\|x\|^{2}}.
	\end{equation} 
	On the other hand let $x\in\tilde{D}$ with $\|x\|=1$, then 
	$$\frac{-1}{~2}x^TA_jx=\frac{-1}{~2}\,\frac{x^TA_jx}{\|x\|^{2}}\geq\inf_{x\neq0} \frac{-1}{~2}\frac{x^TA_jx}{\|x\|^{2}}.$$
	Therefore
	\begin{equation}\label{ine2}
		\inf_{\Vert x \Vert =1} \frac{-1}{~2}x^TA_jx\geq\inf_{x\neq0} \frac{-1}{~2}\frac{x^TA_jx}{\|x\|^{2}}.
	\end{equation}
	From (\ref{ine1}) and (\ref{ine2}), the result follows.
\end{proof}
The following proposition describes the character of constraint and objective functions.
\begin{proposition}\label{p2}
	The function 
	$$g_{j}(x)=f_j(x)-\frac{1}{2}x^TA_jx $$ is $\mu_j$-convex, for all $j\in\{0,1,2,3,\cdots,m\}$, where $\mu_j$ is given by (\ref{mu}).
\end{proposition}
\begin{proof}
	Let $j\in\{0,1,2,3,\cdots,m\}$ and by the above Proposition (\ref{p1}), we have
	$$\mu_j=\inf_{x\neq0}\frac{-1}{~2}\frac{x^TA_jx}{\|x\|^{2}}.$$
	Therefore $$x^T\mu_jx=\mu_j\|x\|^{2}\leq\frac{-1}{~2}x^TA_jx,~~\forall\,x\in\tilde{D}\smallsetminus\{0\}$$
	and it trivially follows for $x=0$. So
	$$-\frac{1}{2}x^TA_jx-x^T\mu_jx\geq0,~~\forall\,x\in\tilde{D}.$$
	That is, the quadratic form $\displaystyle-\frac{1}{2}x^TA_jx-x^T\mu_jx$ is positive semi-definite. So $$\displaystyle f_j(x)-\frac{1}{2}x^TA_jx-\mu_j\|x\|^{2}$$ is convex, as $f_j(x)$ is convex. Hence $$g_{j}(x)=f_j(x)-\frac{1}{2}x^TA_jx $$ is $\mu_j$-convex.
\end{proof}
The following proposition is essential to derive the sufficient global optimality condition.
\begin{proposition}\label{p3}
	If $h:C\subseteq \R^n\;\rightarrow \R$ is $\rho$-convex over $C$ then
	$$h(x)-h(\bar{x})\ge \rho \Vert x-\bar{x} \Vert^2 + (\nabla h(x))^T(x-\bar{x}),~~\forall\,x,\bar{x} \in C.$$
\end{proposition}
\begin{proof}
	Let $x,\;\bar{x} \in C.$ Then
	\[
	h(x)-h(\bar{x})= (h(x)-\rho \Vert x\Vert^2)-(h(\bar{x})-\rho \Vert \bar{x}\Vert^2) +\rho \Vert x\Vert^2-\rho \Vert \bar{x} \Vert^2.
	\] 
	Convexity of $h(x)-\rho \Vert x\Vert$ yields
	$$ (h(x)-\rho \Vert x\Vert^2)-(h(\bar{x})-\rho \Vert \bar{x}\Vert^2) \ge (\nabla h (\bar{x})-2\rho\bar{x})^T(x-\bar{x}).$$
	Therefore,
	\begin{eqnarray*}
		h(x)-h(\bar{x}) &\ge&
		(\nabla h (\bar{x})-2\rho\bar{x})^T(x-\bar{x})+\rho \Vert x\Vert^2-\rho \Vert \bar{x} \Vert^2\\
		&= &\rho \Vert x-\bar{x} \Vert^2 + (\nabla h(x))^T(x-\bar{x}).
	\end{eqnarray*}
	Hence the conclusion follows.
\end{proof}
\subsection{Karush Kuhn Tucker (KKT) Necessary Conditions:} First we recall the standard Karush Kuhn Tucker (KKT) conditions \cite{KKT} for the following Mathematical Programing Problem:
\begin{eqnarray*}
\mbox{(MPP)~~}~~\min_{x\in\mathbb{R}^{n}} f(x)&~& \\
\mbox{s.t.~~}~~ c_{i}(x) &=&0,~~\forall~i\in E\\
c_{i}(x) &\geq&0,~~\forall~i\in I;
 \end{eqnarray*}
where $E$ and $I$ are some finite index sets and the objective function $f$ and the constraints $c_i$ are continuously differentiable.
The point $x$ is called a \textit{regular point} of (MPP) if and only if the set $\{\nabla c_{i}(x),~\mid~i\in\Lambda(x)\}$ is linearly independent; where $\Lambda(x)=\{i\in E\cup I~\mid~c_{i}(x)=0\}$.\\
\noindent
\underline{Standard KKT Condition}: If a regular point $\overline{x}$ is a local minimizer of (MPP), then there exist multipliers 
$\lambda_{i},~i\in E\cup I$ such that
\begin{center}
\begin{itemize}
	\item $\displaystyle\nabla_{x} L(\overline{x},\lambda)=\nabla f(\overline{x})-\sum_{i\in E\cup I}\lambda_{i}\nabla c_{i}(x)=0$
	\item $c_{i}(x)=0,~~\mbox{~for~}~i\in E,~~c_{i}(x)\geq0,~~\mbox{~for~}~~i\in I$
	\item $\lambda_{i}\geq0,~~\mbox{~for~}~~i\in I$
	\item $\lambda_{i}c_{i}(x)=0,~~\mbox{~for~}~~i\in E\cup I$.
\end{itemize}
\end{center}
The above four conditions are known as Karush Kuhn Tucker (KKT) necessary conditions for (MPP).
\subsection{General Necessary Conditions for Local Minimizers}
In this section analogous KKT necessary condition for model problem (MP) consisting of mixed bound variables, is obtained by 
reformulating (MP) into a standard Mathematical Programming Problem (MPM). The following lemma is the key to obtain such KKT condition.
\begin{lemma}\label{lem1}
If $\overline{x}\in\tilde{D}$ is a local minimizer of (MP), then following optimality condition holds:
\begin{equation}\label{lcl1}
\nabla_{x} L(\overline{x},\lambda)(x-\overline{x})\geq 0, ~\forall\,x\in\Delta;
\end{equation}
where $\lambda=(\lambda_j)\in\mathbb{R}^m$ and $\lambda_j:\,j=1,2,3,\cdots,m$ are the Lagrangian multipliers associated with $\overline{x}\in\tilde{D}$.
\end{lemma}
\begin{proof}
The problem (MP) can be reformulated as a Mathematical Programming Problem (referred: model problem-modified) as follows
\begin{eqnarray*}
\mbox{(MPM)~~}~~\min_{x\in\mathbb{R}^{n}} f_{0}(x)&~&~\\
\mbox{s.t.~~}~~ f_{j}(x) &\leq0,& ~~\forall~j\in\{1,2,3,\cdots,m\}\\
(x_{i}-u_{i})(x_{i}-v_{i})&\leq 0,& ~i\in I \\
(x_{i}-u_{i})(x_{i}-v_{i})&=0,& ~i\in J.
 \end{eqnarray*}
For $(\lambda,\mu,\nu)\in\mathbb{R}^{m}\times\mathbb{R}^{|I|}\times\mathbb{R}^{|J|}$, define the Lagrangian $L^{\ast}(\cdot,\lambda,\mu,\nu)$ of (MPM) by
\begin{equation}\label{lag_MMP}
L^{\ast}(x,\lambda,\mu,\nu):=\sum_{j=0}^{m}\lambda_{j}f_{j}(x)+\sum_{i\in I}\mu_{i}(x_{i}-u_{i})(x_{i}-v_{i})+\sum_{i\in J}\nu_{i}(x_{i}-u_{i})(x_{i}-v_{i})
 \end{equation}
for all $x\in\mathbb{R}^{n}$, where $\lambda_{0}=1$. Now suppose that, $\overline{x}\in\tilde{D}$ is a local minimizer of (MP), then trivially we have $\overline{x}$ 
is a local minimizer of (MPM). By applying the Kuhn-Tucker conditions for (MPM), we have there exists $(\lambda,\mu,\nu)\in\mathbb{R}^{n}\times\mathbb{R}^{|I|}\times\mathbb{R}^{|J|}$
such that\\
\begin{itemize}
	\item [(1)] $\nabla_{x}L^{\ast}(\overline{x},\lambda,\mu,\nu)=0$\\
	\item [(2)] 
	\begin{itemize}
		\item [i.] $f_{j}(\overline{x})\leq0$, for all $j\in\{1,2,3,\cdots,m\}$
		\item [ii.] $(\overline{x}_{i}-u_{i})(\overline{x}_{i}-v_{i})\leq 0$, for all $i\in I$
		\item [iii.] $(\overline{x}_{i}-u_{i})(\overline{x}_{i}-v_{i})= 0$, for all $i\in J$\\
	\end{itemize}
	\item [(3)]
	\begin{itemize}
		\item [i.] $\lambda_{j}\geq 0$, for all $j\in\{1,2,3,\cdots,m\}$
		\item [ii.] $\mu_{i}\geq 0$, for all $i\in I$
		\item [iii.] $\nu_{i}\geq 0$, for all $i\in J$\\
	\end{itemize}
	\item [(4)]
	\begin{itemize}
		\item [i.] $\lambda_{j} f_{j}(\overline{x})=0$, for all $j\in\{1,2,3,\cdots,m\}$
		\item [ii.] $\mu_{i}(\overline{x}_{i}-u_{i})(\overline{x}_{i}-v_{i})= 0$, for all $i\in I$
		\item [iii.] $\nu_{i}(\overline{x}_{i}-u_{i})(\overline{x}_{i}-v_{i})= 0$, for all $i\in J$.\\
	\end{itemize}
\end{itemize}
Now $\nabla_{x}L^{\ast}(\overline{x},\lambda,\mu,\nu)=0$ splits into
$$\left[\nabla_{x}\left(\sum_{j=0}^{m}\lambda_{j}f_{j}(x)\right)+\nabla_{x}\left(\sum_{i\in I}\mu_{i}(x_{i}-u_{i})(x_{i}-v_{i})\right)+\nabla_{x}\left(\sum_{i\in J}\nu_{i}(x_{i}-u_{i})(x_{i}-v_{i})\right)\right]_{x=\overline{x}}=0.$$
That is, if $k\in I$, then $$\nabla_{x}\left(L(\overline{x},\lambda)\right)_{k}+\mu_{k}(\overline{x}_{k}-u_{k})+\mu_{k}(\overline{x}_{k}-v_{k})=0.$$
Multiplying both side of this equation by $(x_{k}-\overline{x}_{k})$, and using the Kuhn-Tucker Conditions (2), (3) and (4), we get, for each $k\in I$,
$$(\nabla_{x} L(\overline{x},\lambda))_{k}(x_{k}-\overline{x}_{k})\geq 0,$$
for all $x=(x_{i})\in\Delta$. Hence the conclusion follows.
\end{proof}
For all three programing problems considered in this note we shall assume that the linear independence constraint qualification is satisfied.
\section{The Quadratic Programming Problem}
The aim of this section is to derive necessary local optimality condition and sufficient global optimality condition to identifying global minimizers among the local minimizers of the following non-convex quadratic programming problem with mixed variables ($P_1$):
\begin{eqnarray*}
	\mbox{($P_1$)~~}~~\min_{x\in\mathbb{R}^{n}} f_{0}(x)&=&\min_{x\in\mathbb{R}^{n}} \frac{1}{2}x^{T}A_{0}x+a_{0}^{T}x+c_{0} \\
	\mbox{s.t.~~}~~ f_{j}(x) &=&\frac{1}{2}x^{T}A_{j}x+a_{j}^{T}x+c_{j}\leq 0,~~\forall~j\in\{1,2,3,\cdots,m\}\\
	&&x_{i} \in [u_{i},v_{i}], ~i\in I \\
	&&x_{i} \in \{u_{i},v_{i}\}, ~i\in J ;
\end{eqnarray*} 
where $I\cap J=\phi$, $I\cup J=\{1,2,3,\cdots,n\}$, $A_{j}=(a_{st}^{(j)})$ is an $n\times n$ symmetric matrix, for all $j\in\{0,1,2,3,\cdots,m\}$, $a_{j}=(a_{r}^{(j)})\in\mathbb{R}^{n}$, $c_{j}\in\mathbb{R}$ and $u_{i},v_{i}\in\mathbb{R}$ with $u_{i}<v_{i},$ for all $i\in \{1,2,3,\cdots,n\}$.\\
The Lagrangian of ($P_1$) becomes as 
\begin{equation}
\label{lag}L(x,\lambda)=\sum_{j=0}^{m}\lambda_{j}\left(\frac{1}{2}x^{T}A_{j}x+a_{j}^{T}x+c_{j}\right),
\end{equation}
where $\lambda_0=1$ and $\lambda_j\in\mathbb{R}^+$.
Model problem ($P_1$) differs from standard quadratic programming problem  (MP) because it allows a set of constraints which are of box constraints in continuous variables and in discrete variables. Model problem ($P_1$) appear in numerous application including electronic circuit design and computational chemistry and combinatorial optimization \cite{fp,mmr}. ($P_1$) covers for instance bivalent optimization problems where $u_{i}=0,~v_{i}=1$ for all $i\in J$ and $I=\emptyset$ and box constraints problems where $J=\emptyset$.
\subsection{Necessary Conditions for Local Minimizers}
Using the above Lemma (\ref{lem1}) analogous KKT necessary condition for a local minimizer of ($P_1$) is provided in the following theorem. This result will be used to identify global minimizers in the next subsection.
\begin{theorem}\label{TlcQP}
If $\overline{x}\in\tilde{D}$ is a local minimizer of ($P_1$), then  
\begin{equation}\label{lcl2QP}
\chi_{i}(\overline{x})\sum_{j=0}^{m}\lambda_{j}(A_{j}\overline{x}+a_{j})_{i}\leq0,~~\forall\,i\in I;
\end{equation}
where $\lambda_{j}\in\mathbb{R}^+; j=1,2,3,\cdots,m$ are the Lagrangian multipliers associated with $\overline{x}\in\tilde{D}$, $\lambda_0=1$, and $\chi_{i}(\overline{x})$ as in (\ref{eq4}).
\end{theorem}
\begin{proof}
Suppose that, $\overline{x}\in\tilde{D}$ is a local minimizer of ($P_1$).\\ 
First we shall show that, for each $i\in I$,
\begin{equation}\label{eq1QP}
\sum_{j=0}^{m}\lambda_{j}(A_{j}\overline{x}+a_{j})_{i}(t-\overline{x}_{i})\geq0,~~\forall\,t\in[u_{i},v_{i}].
\end{equation}
For: let $i\in I$ and $t\in[u_{i},v_{i}]$ and define the vector $x=(x_{k})\in\mathbb{R}^{n}$ such that
$$x_{k}:=\left\{
\begin{array}{ll}
t & \mathrm{if}\ k=i \\
\overline{x}_{k} & \mathrm{if}\ k\neq i.\\
\end{array}%
\right. $$
Then clearly $x\in\Delta=\{x=(x_{i})\in\mathbb{R}^{n}~|~x_{i}\in[u_{i},v_{i}], ~i\in I\}$. Hence by the Lemma (\ref{lem1}), we have
$$\sum_{j=0}^{m}\lambda_{j}(A_{j}\overline{x}+a_{j})_{i}(t-\overline{x}_{i})\geq0,~~\forall\,t\in[u_{i},v_{i}].$$
Let $i\in I$ arbitrary. \\
\textbf{Case-1:} If $\overline{x}_{i}=u_{i}$, then choose $t=v_{i}$. From (\ref{eq1QP}), we have
\begin{center}
	$\displaystyle\chi_{i}(\overline{x})\sum_{j=0}^{m}\lambda_{j}(A_{j}\overline{x}+a_{j})_{i}\leq0$ as $\chi_{i}(\overline{x})=-1$ and $(v_i-\overline{x}_i)>0$.
\end{center}
\textbf{Case-2:} If $\overline{x}_{i}=v_{i}$, then choose $t=u_{i}$. From (\ref{eq1QP}), we have
\begin{center}
	$\displaystyle\chi_{i}(\overline{x})\sum_{j=0}^{m}\lambda_{j}(A_{j}\overline{x}+a_{j})_{i}\leq0$ as $\chi_{i}(\overline{x})=1$ and $(v_i-\overline{x}_i)<0$.
\end{center}
\textbf{Case-3:} If $\overline{x}_i\in (u_{i} ,v_{i} )$, then on the one hand, choose $t=u_{i}$. From (\ref{eq1QP}), we have
\begin{equation}\label{eq2QP}
\sum_{j=0}^{m}\lambda_{j}(A_{j}\overline{x}+a_{j})_{i}\leq0.
\end{equation}
On the other hand, choose $t=v_{i}$, the from (\ref{eq1QP}), we have
\begin{equation}\label{eq3QP}
\sum_{j=0}^{m}\lambda_{j}(A_{j}\overline{x}+a_{j})_{i}\geq0.
\end{equation}
From (\ref{eq2QP}) and (\ref{eq3QP}), we have $$\sum_{j=0}^{m}\lambda_{j}(A_{j}\overline{x}+a_{j})_{i}=0.$$
Thus $$\displaystyle\chi_{i}(\overline{x})\sum_{j=0}^{m}\lambda_{j}(A_{j}\overline{x}+a_{j})_{i}=\left(\sum_{j=0}^{m}\lambda_{j}(A_{j}\overline{x}+a_{j})_{i}\right)^{2}=0$$ as $\displaystyle\chi_{i}(\overline{x})=\sum_{j=0}^{m}\lambda_{j}(A_{j}\overline{x}+a_{j})_{i}=0.$ 
Hence by the summary of all above three cases, the optimality condition (\ref{lcl2QP}) follows.
\end{proof}
\subsection{Sufficient Global Optimality Conditions} 
In this section a useful sufficient global optimality condition is presented for a local minimizer to be a Global minimizer for ($P_1$). 
We establish such criteria by under-estimating the objective function by the Lagrangian of the problem ($P_1$). 

The following theorem enables us to have sufficient condition for the global minimizers. To establish this theorem, we show that the Lagrangian $L(x,\lambda)$ is a minimizing under-estimator of the objective function $f_{0}$ at the local minimizer $\overline{x}$ over the set $\tilde{D}$ under desired conditions.
\begin{theorem}\label{tscq}
Let $\overline{x}\in\tilde{D}$ be a local minimizer,  $\lambda_{j}\in\mathbb{R}^{+};~j\in \{1,2,\cdots,m\}$ be the Lagrangian multipliers associated with $\overline{x}\in\tilde{D}$ and $\lambda_{0}=1$. If
\begin{equation}\label{SCQP}
\sum_{j=0}^{m}\lambda_{j}\left[A_{j}-\mathrm{diag}\left(2\chi_{1}(\overline{x})\frac{(A_{j}\overline{x}+a_{j})_{1}}{(v_{1}-u_{1})},
\cdots,2\chi_{n}(\overline{x})\frac{(A_{j}\overline{x}+a_{j})_{n}}{(v_{n}-u_{n})}\right)\right]\succeq0,\vspace{-0.1cm}
\end{equation}
then $\overline{x}$ is a global minimizer of ($P_1$). Moreover, if 
\begin{equation}\label{SC1QP}
\sum_{j=0}^{m}\lambda_{j}\left[A_{j}-\mathrm{diag}\left(2\chi_{1}(\overline{x})\frac{(A_{j}\overline{x}+a_{j})_{1}}{(v_{1}-u_{1})},
\cdots,2\chi_{n}(\overline{x})\frac{(A_{j}\overline{x}+a_{j})_{n}}{(v_{n}-u_{n})}\right)\right]\succ0,\vspace{-0.1cm}
\end{equation}
then $\overline{x}$ is a unique global minimizer of ($P_1$).
\end{theorem}
\begin{proof}
Let $\overline{x}\in\tilde{D}$ be a local minimizer and take, for each $i\in \{1,2,3,\cdots,n\}$ and $j\in \{0,1,2,\cdots,m\}$,
$$q_{i}^{(j)}=2\chi_{i}(\overline{x})\frac{(A_{j}\overline{x}+a_{j})_{i}}{(v_{i}-u_{i})}.$$
Now from the Kuhn-Tucker conditions for ($P_1$), we have there exists $\lambda_{j}\in\mathbb{R}^{+};~j\in \{0,1,2,\cdots,m\}$ with $\lambda_{0}=1$ such that
$f_{0}(\overline{x})= L(\overline{x},\lambda)$ as $\lambda_jf_j(\overline{x})=$, for all $j=1,2,3,\cdots,m$ and $$f_{0}(x)\geq f_0(x)+\sum_{j=1}^{m}\lambda_jf_j(x)= L(x,\lambda),~~\forall\,x \in\tilde{D}$$
as $\lambda_jf_j(x)\leq0$ for all $x \in\tilde{D}$ and $j=1,2,3,\cdots,m.$
So we have
\begin{equation}\label{eq1QP2}
f_{0}(x)-f_{0}(\overline{x})\geq L(x,\lambda)-L(\overline{x},\lambda),~~\forall\,x \in\tilde{D}.
\end{equation}
Now for any $x \in\tilde{D}$,
\begin{eqnarray*}
L(x,\lambda)-L(\overline{x},\lambda)&=&\sum_{j=0}^{m}\lambda_{j}\left[\left(\frac{1}{2}x^{T}A_{j}x+a_{j}^{T}x+c_{j}\right)-
												\left(\frac{1}{2}\overline{x}^{T}A_{j}\overline{x}+a_{j}^{T}\overline{x}+c_{j}\right)\right]\\
&=&\sum_{j=0}^{m}\frac{1}{2}\lambda_{j}\left[(x-\overline{x})^{T}A_{j}(x-\overline{x})+2(A_{j}\overline{x}+a_{j})^{T}(x-\overline{x})\right]\\
&=&\sum_{j=0}^{m}\frac{1}{2}\lambda_{j}(x-\overline{x})^{T}[A_{j}-
\mathrm{diag}(q_{1}^{(j)},q_{2}^{(j)},\cdots,q_{n}^{(j)})](x-\overline{x})\\
&&+\sum_{j=0}^{m}\frac{1}{2}\lambda_{j}(x-\overline{x})^{T}\mathrm{diag}
(q_{1}^{(j)},q_{2}^{(j)},\cdots,q_{n}^{(j)})(x-\overline{x})\\
&&+\sum_{j=0}^{m}\lambda_{j}(A_{j}\overline{x}+a_{j})^{T}(x-\overline{x}).
\end{eqnarray*}
Thus 
\begin{eqnarray*}
	L(x,\lambda)-L(\overline{x},\lambda)&=&\sum_{j=0}^{m}\frac{1}{2}\lambda_{j}(x-\overline{x})^{T}[A_{j}-
	\mathrm{diag}(q_{1}^{(j)},q_{2}^{(j)},\cdots,q_{n}^{(j)})](x-\overline{x})\\
	&&+\sum_{j=0}^{m}\lambda_{j}\left[\sum_{i=1}^{n}\frac{1}{2}q_{i}^{(j)}(x_{i}-\overline{x}_{i})^{2}+(A_{j}\overline{x}+a_{j})_{i}(x_{i}-\overline{x}_{i})\right].
\end{eqnarray*}
From the condition (\ref{SCQP}), we have $[A_{j}-\mathrm{diag}(q_{1}^{(j)},q_{2}^{(j)},\cdots,q_{n}^{(j)})]\succeq0$.
Thus, for all $x \in\tilde{D}$, $$\sum_{j=0}^{m}\frac{1}{2}\lambda_{j}(x-\overline{x})^{T}[A_{j}-
\mathrm{diag}(q_{1}^{(j)},q_{2}^{(j)},\cdots,q_{n}^{(j)})](x-\overline{x})\geq0.$$
Therefore $$L(x,\lambda)-L(\overline{x},\lambda)\geq\sum_{j=0}^{m}\lambda_{j}\left[\sum_{i=1}^{n}\frac{1}{2}q_{i}^{(j)}(x_{i}-\overline{x}_{i})^{2}+(A_{j}\overline{x}+a_{j})_{i}(x_{i}-\overline{x}_{i})\right],$$ for all $x \in\tilde{D}$. Now using the optimality condition (\ref{lcl2QP}), we show that  
$$\sum_{j=0}^{m}\lambda_{j}\left[\sum_{i=1}^{n}\frac{1}{2}q_{i}^{(j)}(x_{i}-\overline{x}_{i})^{2}+(A_{j}\overline{x}+a_{j})_{i}(x_{i}-\overline{x}_{i})\right]\geq0,$$ for all $x \in\tilde{D}$. For let $x=(x_{i})\in\tilde{D}$ arbitrarily.\\
\textbf{Case-1:} If $\overline{x}_{i}=u_{i}$, then (\ref{lcl2QP}) gives $\displaystyle\sum_{j=0}^{m}\lambda_{j}(A_{j}\overline{x}+a_{j})_{i}\geq0$.
Now 
\begin{eqnarray*}
~&&\sum_{j=0}^{m}\lambda_{j}\left[\frac{1}{2}q_{i}^{(j)}(x_{i}-\overline{x}_{i})^{2}+(A_{j}\overline{x}+a_{j})_{i}(x_{i}-\overline{x}_{i})\right]\\
&=&\sum_{j=0}^{m}\lambda_{j}\left[(-1)\frac{(A_{j}\overline{x}+a_{j})_{i}}{(v_{i}-u_{i})}(x_{i}-\overline{x}_{i})^{2}+(A_{j}\overline{x}+a_{j})_{i}(x_{i}-\overline{x}_{i})\right]\\
&=&(x_{i}-u_{i})\left[\sum_{j=0}^{m}\lambda_{j}\frac{(A_{j}\overline{x}+a_{j})_{i}}{(v_{i}-u_{i})}\right]\left(1-\frac{(x_{i}-u_{i})}{(v_{i}-u_{i})}\right).
\end{eqnarray*}
Therefore 
$$\sum_{j=0}^{m}\lambda_{j}\left[\frac{1}{2}q_{i}^{(j)}(x_{i}-\overline{x}_{i})^{2}+(A_{j}\overline{x}+a_{j})_{i}(x_{i}-\overline{x}_{i})\right]$$
is non-negative, if $i\in I$ and is equal to $0$, if $i\in J$.\\
\textbf{Case-2:} If $\overline{x}_{i}=v_{i}$, then (\ref{lcl2QP}) gives  $\displaystyle\sum_{j=0}^{m}\lambda_{j}(A_{j}\overline{x}+a_{j})_{i}\leq0$ and in a similar method of Case-1, we will have
$$\sum_{j=0}^{m}\lambda_{j}\left[\frac{1}{2}q_{i}^{(j)}(x_{i}-\overline{x}_{i})^{2}+(A_{j}\overline{x}+a_{j})_{i}(x_{i}-\overline{x}_{i})\right]$$
is non-negative, if $i\in I$ and is equal to $0$, if $i\in J$.\\ 
\textbf{Case-3:} If $\overline{x}_{i}\in(u_{i},v_{i})$ and $i\in I$, then $\displaystyle\sum_{j=0}^{m}\lambda_{j}(A_{j}\overline{x}+a_{j})_{i}=0$
 and so $\displaystyle\sum_{j=0}^{m}\lambda_{j}q_{i}^{(j)}=0$. Thus
$$\sum_{j=0}^{m}\lambda_{j}\left[\frac{1}{2}q_{i}^{(j)}(x_{i}-\overline{x}_{i})^{2}+(A_{j}\overline{x}+a_{j})_{i}(x_{i}-\overline{x}_{i})\right]=0.$$
The summary of the above all three cases gives us that
$$\sum_{j=0}^{m}\lambda_{j}\left[\sum_{i=1}^{n}\frac{1}{2}q_{i}^{(j)}(x_{i}-\overline{x}_{i})^{2}+(A_{j}\overline{x}+a_{j})_{i}(x_{i}-\overline{x}_{i})\right]\geq0.$$
Thereby $$f_{0}(x)-f_{0}(\overline{x})\geq L(x,\lambda)-L(\overline{x},\lambda)\geq0,$$
for all $x \in\tilde{D}$ and so $L(x,\lambda)$ is a minimizing under-estimator of $f_{0}$. Hence $\overline{x}\in\tilde{D}$ is a Global minimizer of ($P_1$). If the condition (\ref{SC1QP}) holds, then we can show that 
$$f_{0}(x)-f_{0}(\overline{x})\geq L(x,\lambda)-L(\overline{x},\lambda)>0,$$
for all $x \in\tilde{D}$. Therefore the uniqueness immediately follows.
\end{proof}
Now we shall illustrate the sufficiency criteria given by above Theorem.
\begin{example}
Consider the following quadratic non-convex minimization problem ($E_1$):
\begin{eqnarray*}
	\mbox{($E_1$)~~}~~\min_{(x_1,x_2)^\intercal\in\mathbb{R}^{2}} f\begin{pmatrix}
		x_1 \\ x_2
	\end{pmatrix}&=&\min_{(x_1,x_2)^\intercal\in\mathbb{R}^{2}} x_1^2+x_2^2-4x_1x_2-8x_1-8x_2+32 \\
	\mbox{s.t.~~}~~ g\begin{pmatrix}
		x_1 \\ x_2
	\end{pmatrix} &=&x_1^2+x_1x_2-x_1+3x_2-12\leq 0,\\
	&&x_1\in [-2,2],\\
	&&x_2\in \{-2,2\}.
\end{eqnarray*}
We can write $f\begin{pmatrix}
x_1 \\ x_2
\end{pmatrix}$ and $g\begin{pmatrix}
x_1 \\ x_2
\end{pmatrix}$ in the form of
$$f(\mathbf{x})=\dfrac{1}{2}\mathbf{x}^\intercal A_0\mathbf{x}+\mathbf{a}_0^\intercal\mathbf{x}+c_0\mbox{~~and~~}g(\mathbf{x})=\dfrac{1}{2}\mathbf{x}^\intercal A_1\mathbf{x}+\mathbf{a}_1^\intercal\mathbf{x}+c_1$$  respectively, where $\mathbf{x}=\begin{pmatrix}
x_1 \\ x_2
\end{pmatrix}$, $A_0=\begin{bmatrix}
~\,\,2 & -4\\ -4 & ~\,\,2
\end{bmatrix}$, $A_1=\begin{bmatrix}
2 & 1\\ 1 & 0
\end{bmatrix}$, $\mathbf{a}_0=\begin{pmatrix}
-8 \\ -8
\end{pmatrix}$, $\mathbf{a}_1=\begin{pmatrix}
-1 \\ ~\,\,3
\end{pmatrix}$, $c_0=32$ and $c_1=-12$. Let $u_1=u_2=-2$ and $v_1=v_2=2$.
The Lagrangian of ($E_1$) is 
\begin{equation}\label{lagE1}
L(\mathbf{x},\lambda,\mu,\nu)=x_1^2+x_2^2-4x_1x_2-8x_1-8x_2+32+\lambda(x_1^2+x_1x_2-x_1+3x_2-12)+\mu(x_1^2-4)+\nu(x_2^2-4).
\end{equation}
Applying the KKT conditions for ($E_1$) we get
\begin{itemize}
	\item[(1)] $\dfrac{\partial L(\mathbf{x},\lambda,\mu,\nu)}{\partial x_1}=(2x_1-4x_2-8)+\lambda(2x_1+x_2-1)+2\mu x_1=0$,
	\item[(2)] $\lambda,\mu\geq0$, $x_1^2+x_1x_2-x_1+3x_2-12\leq 0$, $(x_1^2-4)\leq0$,
	\item[(3)] $\lambda(x_1^2+x_1x_2-x_1+3x_2-12)=0$,
	\item[(4)] $\mu(x_1^2-4)=0$,
	\item[(5)] $\nu=0$ and $x_2=-2$ or $x_2=2$.
\end{itemize}
\textbf{Case 1:} $\lambda=0$ and $\mu=0$. In this case either $x_1=4$ or $x_1=10$, but this contradicts with the condition $(x_1^2-4)\leq0$, so there is no solution to KKT in this case.\\\\
\textbf{Case 2:} $\lambda=0$ and $x_1^2-4=0$. In this case for both values $x_1=\pm2$, $\mu<0$. This violates the rule $\mu\geq0$ and there is no solution to KKT in this case.\\\\
\textbf{Case 3:} $(x_1^2+x_1x_2-x_1+3x_2-12)=0$ and $\mu=0$. In this case, if $x_2=2$, $x_1\in\{-3,2\}$ and if $x_2=-2$, $x_1\in\{-3,6\}$. Therefore the possible value for the pair $(x_1, x_2)$ is $(x_1,x_2)=(2,2)$. The solution, in this case, is $(x_1,x_2,\lambda,\mu,\nu)=(2,2,12/5,0,0)$.\\\\
\textbf{Case 4:} $(x_1^2+x_1x_2-x_1+3x_2-12)=0$ and $x_1^2-4=0$. In this case, these two equations are satisfied by only one pair $(x_1,x_2)=(2,2)$ and it repeats the solution of case 3.
 
Therefore, the problem ($E_1$) has a unique local minimizer $\overline{\mathbf{x}}=(2,2)^\intercal$ and the Lagrange multiplier $\lambda=\dfrac{12}{5}$. Notice that the linear independence constraint qualification is satisfied in this problem ($E_1$). Further, direct calculation shows that 
the global optimality condition (\ref{SC1QP}) is satisfied, and therefore $\overline{\mathbf{x}}=(2,2)^\intercal$ is the unique global minimizer.
\end{example}
Let us see  another example, which illustrates the significance of the optimality conditions.
\begin{example}
Consider the following quadratic non-convex minimization problem ($E_2$):
\begin{eqnarray*}
	\mbox{($E_2$)~~}~~\min_{(x_1,x_2)^\intercal\in\mathbb{R}^{2}} f\begin{pmatrix}
		x_1 \\ x_2
	\end{pmatrix}&=&\min_{(x_1,x_2)^\intercal\in\mathbb{R}^{2}} -x_1^2-x_2^2-x_1x_2+x_1\\
	\mbox{s.t.~~}~~ g\begin{pmatrix}
		x_1 \\ x_2
	\end{pmatrix} &=&x_1^2+x_2^2-2\leq 0,\\
	&&x_1\in [-1,1],\\
	&&x_2\in \{-1,1\}.
\end{eqnarray*}
We can write $f\begin{pmatrix}
x_1 \\ x_2
\end{pmatrix}$ and $g\begin{pmatrix}
x_1 \\ x_2
\end{pmatrix}$ in the form of
$$f(\mathbf{x})=\dfrac{1}{2}\mathbf{x}^\intercal A_0\mathbf{x}+\mathbf{a}_0^\intercal\mathbf{x}+c_0\mbox{~~and~~}g(\mathbf{x})=\dfrac{1}{2}\mathbf{x}^\intercal A_1\mathbf{x}+\mathbf{a}_1^\intercal\mathbf{x}+c_1$$  respectively, where $\mathbf{x}=\begin{pmatrix}
x_1 \\ x_2
\end{pmatrix}$, $A_0=\begin{bmatrix}
-2 & -1\\ -1 & -2
\end{bmatrix}$, $A_1=\begin{bmatrix}
2 & 0\\ 0 & 2
\end{bmatrix}$, $\mathbf{a}_0=\begin{pmatrix}
1 \\ 0
\end{pmatrix}$, $\mathbf{a}_1=\begin{pmatrix}
0 \\ 0
\end{pmatrix}$, $c_0=0$ and $c_1=-2$. Let $u_1=u_2=-1$ and $v_1=v_2=1$.
The Lagrangian of ($E_2$) is 
\begin{equation}\label{lagE2}
L(\mathbf{x},\lambda,\mu,\nu)=-x_1^2-x_2^2-x_1x_2+x_1+\lambda(x_1^2+x_2^2-2)+\mu(x_1^2-1)+\nu(x_2^2-1).
\end{equation}
Apply the KKT conditions for ($E_2$). Then we have
\begin{itemize}
	\item[(1)] $\dfrac{\partial L(\mathbf{x},\lambda,\mu,\nu)}{\partial x_1}=(-2x_1-x_2+1)+2\lambda x_1+2\mu x_1=0$,
	\item[(2)] $\lambda,\mu\geq0$, $x_1^2+x_2^2-2\leq 0$, $(x_1^2-4)\leq0$,
	\item[(3)] $\lambda(x_1^2+x_2^2-2)=0$,
	\item[(4)] $\mu(x_1^2-1)=0$,
	\item[(5)] $\nu=0$ and $x_2=-1$ or $x_2=1$.
\end{itemize}
\textbf{Case 1:} $\lambda\ne0$ or $\mu\ne0$. In this case, there are two sets of solutions $\{(x_1,x_2,\lambda,\mu,0)~\mid~x_1=\pm1, x_2=1,\lambda+\mu=1~~\mbox{and}~~\lambda,\mu\geq0\}$ and $\{(-1,-1,\lambda,\mu,0)~\mid~\lambda+\mu=2~~\mbox{and}~~\lambda,\mu\geq0\}$ to KKT, for the problem ($E_2$).\\\\
\textbf{Case 2:} $\lambda=0$ and $\mu=0$. In this case, there are two solutions $(x_1,x_2,\lambda,\mu,0)=(1,-1,0,0,0)$ and $(0,1,0,0,0)$.

Let $\overline{\mathbf{x}}_1=(1,1)^\intercal,\overline{\mathbf{x}}_2=(-1,1)^\intercal,\overline{\mathbf{x}}_3=(-1,-1)^\intercal,\overline{\mathbf{x}}_4=(1,-1)^\intercal$ and $\overline{\mathbf{x}}_5=(0,1)^\intercal$. Direct calculations show that $\overline{\mathbf{x}}_3=(-1,-1)^\intercal$ satisfies the condition (\ref{SC1QP}). Other local minimizers do not satisfy any of the two conditions (\ref{SCQP}) or (\ref{SC1QP}) for all $\lambda\in[0,2]$. Therefore, the point $\overline{\mathbf{x}}_3=(-1,-1)^\intercal$ is the unique global minimizer.
\end{example}
\section{$\rho$-Convex Programing Problem}
Development of global optimality conditions to identify global minimizers of various classes of non-convex problems has been the recent trend in non-linear optimization. In similar manner in this section we intend to distinguish local and global  minimizers for the following $\rho-$convex model problem.
\begin{eqnarray*}
	\mbox{($P_2$)~~}~~\min_{x\in\mathbb{R}^{n}} g_{0}(x)&=&\min_{x\in\mathbb{R}^{n}} f_{0}(x)-\frac{1}{2}x^{T}A_{0}x \\
	\mbox{subject to~~}~~ g_{j}(x) &=&f_{j}(x)-\frac{1}{2}x^{T}A_{j}x\leq 0,~~\forall~j\in\{1,2,3,\cdots,m\}\\
	&&x_{i} \in [u_{i},v_{i}], ~i\in I \\
	&&x_{i} \in \{u_{i},v_{i}\}, ~i\in J ;
\end{eqnarray*}
where $I\cap J=\phi$, $I\cup J=\{1,2,3,\cdots,n\}$, for each $j\in\{0,1,2,3,\cdots,m\}$, the function $f_j:\R^n \rightarrow \R$ is convex and  $A_{j}=(a_{st}^{(j)})$ is an $n\times n$ symmetric matrix,  
and $u_{i},v_{i}\in\mathbb{R}$ with $u_{i}<v_{i},$ for all $i\in \{1,2,3,\cdots,n\}$.\\
Recent research displays the development of conditions necessary or sufficient for characterizing global minimizers of smooth functions with bounded mixed variables (see \cite{jrwbox} and other references therein). However, a drawback of this development is that the conditions were neither based on local optimality conditions nor expressed in terms of local minimizers. A verifiable sufficient global optimality condition is established in this paper by refining the method of proof developed in \cite{jls}, and by incorporating the local optimality conditions. Such developed condition is a useful criterion for distinguishing global minimizers among the local minimizers. 
\subsection{Local necessary optimality condition for ($P_2$)}
Using the Lemma (\ref{lem1}), analogous KKT necessary condition for a local minimizer of ($P_2$) is obtained in the following theorem. This result will be used to identify global minimizers in the next section.
\begin{theorem}\label{TlcRC}
	If $\overline{x}\in\tilde{D}$ is a local minimizer of ($P_2$), then  
	\begin{equation}\label{lcl2RC}
	\chi_{i}(\overline{x})\sum_{j=0}^{m}\lambda_{j}(\nabla f_{j}(\overline{x})-A_{j}\overline{x})_{i}\leq0,~~\forall\,i\in I;
	\end{equation}
	where $\lambda_{j}\in\mathbb{R}^+; j=1,2,3,\cdots,m$ are the Lagrangian multipliers associated with $\overline{x}\in\tilde{D}$ and $\lambda_0=1$.
\end{theorem}
\begin{proof}
	Suppose that, $\overline{x}\in\tilde{D}$ is a local minimizer of ($P_2$).\\ 
	First we shall show that, for each $i\in I$,
	\begin{equation}\label{eq1RC}
	\sum_{j=0}^{m}\lambda_{j}(\nabla f_{j}(\overline{x})-A_{j}\overline{x})_{i}(t-\overline{x}_{i})\geq0,~~\forall\,t\in[u_{i},v_{i}].
	\end{equation}
	For: let $i\in I$ and $t\in[u_{i},v_{i}]$ and define the vector $x=(x_{k})\in\mathbb{R}^{n}$ such that
	$$x_{k}:=\left\{
	\begin{array}{ll}
	t & \mathrm{if}\ k=i \\
	\overline{x}_{k} & \mathrm{if}\ k\neq i.\\
	\end{array}%
	\right. $$
	Then clearly $x\in\Delta=\{x=(x_{i})\in\mathbb{R}^{n}~|~x_{i}\in[u_{i},v_{i}], ~i\in I\}$. Hence by the Lemma (\ref{lem1}), we have
	$$\sum_{j=0}^{m}\lambda_{j}(\nabla f_{j}(\overline{x})-A_{j}\overline{x})_{i}(t-\overline{x}_{i})\geq0,~~\forall\,t\in[u_{i},v_{i}].$$
	Let $i\in I$ arbitrary. \\
	\textbf{Case-1:} If $\overline{x}_{i}=u_{i}$, then choose $t=v_{i}$. From (\ref{eq1RC}), we have
	$$\chi_{i}(\overline{x})\sum_{j=0}^{m}\lambda_{j}(\nabla f_{j}(\overline{x})-A_{j}\overline{x})_{i}\leq0.$$
	\textbf{Case-2:} If $\overline{x}_{i}=v_{i}$, then choose $t=u_{i}$. From (\ref{eq1RC}), we have
	$$\chi_{i}(\overline{x})\sum_{j=0}^{m}\lambda_{j}(\nabla f_{j}(\overline{x})-A_{j}\overline{x})_{i}\leq0.$$
	\textbf{Case-3:} If $\overline{x}_i\in (u_{i} ,v_{i} )$, then on the one hand, choose $t=u_{i}$. From (\ref{eq1RC}), we have
	\begin{equation}\label{eq2RC}
	\sum_{j=0}^{m}\lambda_{j}(\nabla f_{j}(\overline{x})-A_{j}\overline{x})_{i}\leq0.
	\end{equation}
	On the other hand, choose $t=v_{i}$, the from (\ref{eq1RC}), we have
	\begin{equation}\label{eq3RC}
	\sum_{j=0}^{m}\lambda_{j}(\nabla f_{j}(\overline{x})-A_{j}\overline{x})_{i}\geq0.
	\end{equation}
	From (\ref{eq2RC}) and (\ref{eq3RC}), we have $$\chi_{i}(\overline{x})=\sum_{j=0}^{m}\lambda_{j}(\nabla f_{j}(\overline{x})-A_{j}\overline{x})_{i}=0.$$
	Thus\vspace{-0.3cm} $$\chi_{i}(\overline{x})\sum_{j=0}^{m}\lambda_{j}(\nabla f_{j}(\overline{x})-A_{j}\overline{x})_{i}=\left(\sum_{j=0}^{m}\lambda_{j}(\nabla f_{j}(\overline{x})-A_{j}\overline{x})_{i}\right)^{2}=0.$$
	Hence by the summary of all above three cases, the optimality condition (\ref{lcl2RC}) follows.
\end{proof}
\begin{remark}
	The Lemma (\ref{lem1}) gives us, for given local minimizer $\overline{x}\in\tilde{D}$ the set 
	$$\left\{\frac{(\nabla L(\overline{x},\lambda))^{T}(x-\overline{x})}{\|x-\overline{x}\|^{2}}~\mid ~x \in \tilde{D} ~\mbox{~with~}~x\neq\overline{x}\right\}$$ is bounded below by 0.
	Hence, $$\inf\left\{\frac{(\nabla L(\overline{x},\lambda))^{T}(x-\overline{x})}{\|x-\overline{x}\|^{2}}~\mid ~x \in \tilde{D} ~\mbox{~with~}~x\neq\overline{x}\right\}~\mbox{~exists.~}~$$
\end{remark}
\subsection{A Sufficiency Global Optimality Condition  }
The following theorem provides sufficient condition for the global minimizers. But unfortunately this condition might not be readily verifiable as it involved with another minimization problem.
\begin{theorem}\label{tsc}
	Let $\overline{x}\in\tilde{D}$ be a local minimizer of ($P_2$) and $\lambda_{0}=1$. Suppose that $\lambda_{j}\in\mathbb{R}^{+};~j\in \{1,2,\cdots,m\}$ are the Lagrangian multipliers associated with $\overline{x}\in\tilde{D}$ as given in the Lemma (\ref{lem1}) and $\mu_{j}\in\R$ are the first eigenvalues of the symmetric matrices $A_{j}$, for all $j\in\{0,1,2,\cdots,m\}$. If 
	\begin{equation}\label{SC}
	\sum_{j=0}^m\lambda_j\mu_j\geq-\inf_{x\neq\overline{x}}\frac{(\nabla L(\overline{x},\lambda))^{T}(x-\overline{x})}{\|x-\overline{x}\|^{2}},
	\end{equation}
	then $\overline{x}$ is a global minimizer of ($P_2$).
\end{theorem}
\begin{proof}
	Let $\overline{x} \in \tilde{D}$ be a local minimizer of ($P_2$). Now for any $x \in \tilde{D}$,
	\begin{eqnarray*}
		g_0(x)-g_0(\bar x) &\ge& g_0(x)+\sum_{j=1}^{m}\lambda_jg_j(x,\lambda)-g_0(\bar{x})\\
		&=&L(x,\lambda)-g_0(\bar{x})-\sum_{j=1}^{m}\lambda_jg_j(\bar{x},\lambda)\\
		&=&L(x,\lambda)-L(\bar{x},\lambda)
	\end{eqnarray*}
	and we have from the Proposition (\ref{p2}) the Lagrangian 
	$$\displaystyle L(x,\lambda)=\sum_{j=0}^{m}\lambda_{j}\left(f_j(x)-\frac{1}{2}x^TA_jx\right)$$ is $\rho-$convex, as $L(x,\lambda)-\rho\|x\|^{2}$ is a convex function; where $\displaystyle \rho=\sum_{j=0}^m\lambda_j\mu_j$. Hence by the proposition (\ref{p3}) we have 
	\begin{equation}\label{sceq1RC*}
	L(x,\lambda)-L(\bar{x},\lambda)\ge \rho \Vert x-\bar{x} \Vert^2 + (\nabla L(\bar{x},\lambda))^T (x-\bar{x}).
	\end{equation}
	Thus 
	\begin{equation}\label{sceq1RC**}
	g_0(x)-g_0(\bar x) \ge \rho \Vert x-\bar{x} \Vert^2 + (\nabla L(\bar{x},\lambda))^T (x-\bar{x}),~~\forall\, x \in \tilde D.
	\end{equation}
	The condition (\ref{SC}) implies, 
	$$\rho=\sum_{j=0}^m\lambda_j\mu_j\geq -\frac{(\nabla L(\overline{x},\lambda))^{T}(x-\overline{x})}{\|x-\overline{x}\|^{2}},$$
	for all $x \in \tilde{D}$ with $x\neq\overline{x}$. Therefore we have 
	$$\rho \Vert x-\bar{x} \Vert^2 + (\nabla L(\bar{x},\lambda))^T (x-\bar{x})\geq0,~~\forall\, x \in \tilde D$$
	and so $$g_0(x)-g_0(\bar x)\geq0,~~\forall\, x \in \tilde D.$$
	Hence the theorem follows.
\end{proof}
The following verifiable sufficient global optimality condition is the main result of this section which provides sufficiency criterion for a local minimizer to be a global minimizer for ($P_2$).
\subsection{A verifiable sufficient optimality conditions for ($P_2$)}
\begin{theorem}\label{tsc1}
	Let $\overline{x}\in\tilde{D}$ be a local minimizer of ($P_2$) and $\lambda_{0}=1$. Suppose that $\lambda_{j}\in\mathbb{R}^{+};~j\in \{1,2,\cdots,m\}$ are the Lagrangian multipliers associated with $\overline{x}\in\tilde{D}$ as given in the Lemma (\ref{lem1}) and $\mu_{j}\in\R$ are the first eigenvalues of the symmetric matrices $\displaystyle-\frac{1}{2}A_j$, for all $j\in\{0,1,2,\cdots,m\}$. If 
	\begin{equation}\label{SC1RC}
	\sum_{j=0}^{m}\lambda_{j}\sum_{i=1}^{n}\left[\frac{\chi_{i}(\overline{x})(\nabla f_{j}(\overline{x})-A_{j}\overline{x})_{i}}{(v_{i}-u_{i})}\right]\leq\sum_{j=0}^{m}\lambda_{j}\mu_{j},
	\end{equation}
	then $\overline{x}$ is a global minimizer of ($P_2$).
\end{theorem}
\begin{proof}
	Let $\overline{x} \in \tilde{D}$ be a local minimizer of ($P_2$). Now for any $x \in \tilde{D}$,
	\begin{eqnarray*}
		g_0(x)-g_0(\bar x) &\ge& g_0(x)+\sum_{j=1}^{m}\lambda_jg_j(x)-g_0(\bar{x})\\
		&=&L(x,\lambda)-g_0(\bar{x})-\sum_{j=1}^{m}\lambda_jg_j(\bar{x})\\
		&=&L(x,\lambda)-L(\bar{x},\lambda)
	\end{eqnarray*}
	and
	\begin{eqnarray*}
		L(x,\lambda)-L(\bar{x},\lambda)&=& \sum_{j=0}^{m}\lambda_{j}\left[\left(f_j(x)-\frac{1}{2}x^TA_jx\right)-\left(f_j(\overline{x})-\frac{1}{2}\overline{x}^TA_j\overline{x}\right)\right]\\
		&=& \sum_{j=0}^{m}\lambda_{j}\left(f_j(x)-f_j(\overline{x})\right)- \sum_{j=0}^{m}\frac{1}{2}\lambda_{j}\left(x^TA_jx-\overline{x}^TA_j\overline{x}\right)\\
		&=& \sum_{j=0}^{m}\lambda_{j}\left(f_j(x)-f_j(\overline{x})\right)-\sum_{j=0}^{m}\lambda_{j}\left[\frac{1}{2}(x-\overline{x})^{T}A_{j}(x-\overline{x})+(A_{j}\overline{x})^{T}(x-\overline{x})\right].
	\end{eqnarray*}
	But for all $j\in\{0,1,2,\cdots,m\}$, the convexity of $f_{j}$ implies
	$$f_{j}(x)-f_{j}(\bar{x})\ge (\nabla f_{j}(\overline{x}))^T(x-\bar{x}),~~\forall\,x \in \tilde{D}$$
	and from the definition of $\mu_{j}$ we have
	$$-\frac{1}{2}(x-\overline{x})^{T}A_{j}(x-\overline{x})\ge\mu_{j}\|x-\overline{x}\|^{2},~~~\forall\,x \in \tilde{D}.$$
	Thus for all $x \in \tilde{D}$,
	\begin{eqnarray*}
		L(x,\lambda)-L(\bar{x},\lambda)&\ge& \sum_{j=0}^{m}\lambda_{j}\left( (\nabla f_{j}(\overline{x}))^T(x-\bar{x})\right)\\
		&&+\sum_{j=0}^{m}\lambda_{j}\left[\mu_{j}\|x-\overline{x}\|^{2}-(A_{j}\overline{x})^{T}(x-\overline{x})\right]\\
		&=&\sum_{j=0}^{m}\lambda_{j}[(\nabla f_{j}(\overline{x}))-A_{j}\overline{x})^{T}(x-\overline{x})+\mu_{j}\|x-\overline{x}\|^{2}].\\
	\end{eqnarray*}
Therefore
\begin{eqnarray*}
	L(x,\lambda)-L(\bar{x},\lambda)&\ge&\sum_{j=0}^{m}\lambda_{j}\sum_{i=1}^{n}\left[(\nabla f_{j}(\overline{x}))-A_{j}\overline{x})_{i}(x_{i}-\overline{x}_{i})+\mu_{j}(x_{i}-\overline{x}_{i})^{2}\right]\\
	&=&\sum_{j=0}^{m}\lambda_{j}\sum_{i=1}^{n}\left[(\nabla f_{j}(\overline{x}))-A_{j}\overline{x})_{i}(x_{i}-\overline{x}_{i})+\frac{\chi_{i}(\overline{x})(\nabla f_{j}(\overline{x})-A_{j}\overline{x})_{i}}{(v_{i}-u_{i})}(x_{i}-\overline{x}_{i})^{2}\right]\\
	&&+\sum_{j=0}^{m}\lambda_{j}\sum_{i=1}^{n}\left[\mu_{j}-\frac{\chi_{i}(\overline{x})(\nabla f_{j}(\overline{x})-A_{j}\overline{x})_{i}}{(v_{i}-u_{i})}\right](x_{i}-\overline{x}_{i})^{2}.
\end{eqnarray*}
	The condition (\ref{SC1RC}), gives that 
	$$\sum_{j=0}^{m}\lambda_{j}\sum_{i=1}^{n}\left[\mu_{j}-\frac{\chi_{i}(\overline{x})(\nabla f_{j}(\overline{x})-A_{j}\overline{x})_{i}}{(v_{i}-u_{i})}\right](x_{i}-\overline{x}_{i})^{2}\ge0,~~\forall\,x \in \tilde{D}.$$
	Now for any $x=(x_{i})\in\tilde{D}$ and for all $i\in\{1,2,\cdots,n\}$,
	$$\sum_{j=0}^{m}\lambda_{j}\left[(\nabla f_{j}(\overline{x}))-A_{j}\overline{x})_{i}(x_{i}-\overline{x}_{i})+\frac{\chi_{i}(\overline{x})(\nabla f_{j}(\overline{x})-A_{j}\overline{x})_{i}}{(v_{i}-u_{i})}(x_{i}-\overline{x}_{i})^{2}\right]\ge0.$$
	To see this, let $x=(x_{i})\in\tilde{D}$ and $i\in\{1,2,\cdots,n\}$.\\
	\textbf{Case-1:} If $\overline{x}_{i}=u_{i}$, then by (\ref{lcl2RC}), we have $\displaystyle\sum_{j=0}^{m}\lambda_{j}(\nabla f_{j}(\overline{x})-A_{j}\overline{x})_{i}\geq0,~~\forall\,i\in I$.
	Now 
	\begin{eqnarray*}
		~&&\sum_{j=0}^{m}\lambda_{j}\left[(\nabla f_{j}(\overline{x}))-A_{j}\overline{x})_{i}(x_{i}-\overline{x}_{i})+\frac{\chi_{i}(\overline{x})(\nabla f_{j}(\overline{x})-A_{j}\overline{x})_{i}}{(v_{i}-u_{i})}(x_{i}-\overline{x}_{i})^{2}\right]\\
		&=&\sum_{j=0}^{m}\lambda_{j}\left[(\nabla f_{j}(\overline{x})-A_{j}\overline{x})_{i}(x_{i}-\overline{x}_{i})+(-1)\frac{(\nabla f_{j}(\overline{x})-A_{j}\overline{x})_{i}}{(v_{i}-u_{i})}(x_{i}-\overline{x}_{i})^{2}\right]\\
		&=&(x_{i}-u_{i})\left[\sum_{j=0}^{m}\lambda_{j}(\nabla f_{j}(\overline{x})-A_{j}\overline{x})_{i}\right]\left(1-\frac{(x_{i}-u_{i})}{(v_{i}-u_{i})}\right).
	\end{eqnarray*}
	Therefore \vspace{-0.2cm}
	$$\sum_{j=0}^{m}\lambda_{j}\left[(\nabla f_{j}(\overline{x}))-A_{j}\overline{x})_{i}(x_{i}-\overline{x}_{i})+\frac{\chi_{i}(\overline{x})(\nabla f_{j}(\overline{x})-A_{j}\overline{x})_{i}}{(v_{i}-u_{i})}(x_{i}-\overline{x}_{i})^{2}\right]$$
	is non-negative, if $i\in I$ and is equal to $0$, if $i\in J$.\\
	\textbf{Case-2:} If $\overline{x}_{i}=v_{i}$, then by (\ref{lcl2RC}), $\displaystyle\sum_{j=0}^{m}\lambda_{j}(\nabla f_{j}(\overline{x})-A_{j}\overline{x})_{i}\leq0,~~\forall\,i\in I$ and in a similar method of Case-1, we will have
	$$\sum_{j=0}^{m}\lambda_{j}\left[(\nabla f_{j}(\overline{x}))-A_{j}\overline{x})_{i}(x_{i}-\overline{x}_{i})+\frac{\chi_{i}(\overline{x})(\nabla f_{j}(\overline{x})-A_{j}\overline{x})_{i}}{(v_{i}-u_{i})}(x_{i}-\overline{x}_{i})^{2}\right]$$
	is non-negative, if $i\in I$ and is equal to $0$, if $i\in J$.\\ 
	\textbf{Case-3:} If $\overline{x}_{i}\in(u_{i},v_{i})$ and $i\in I$, then by (\ref{lcl2RC}), $\displaystyle\sum_{j=0}^{m}\lambda_{j}(\nabla f_{j}(\overline{x})-A_{j}\overline{x})_{i}=0$. Thus
	$$\sum_{j=0}^{m}\lambda_{j}\left[(\nabla f_{j}(\overline{x}))-A_{j}\overline{x})_{i}(x_{i}-\overline{x}_{i})+\frac{\chi_{i}(\overline{x})(\nabla f_{j}(\overline{x})-A_{j}\overline{x})_{i}}{(v_{i}-u_{i})}(x_{i}-\overline{x}_{i})^{2}\right]=0.$$
	The summary of the above all three cases gives us that
	$$\sum_{j=0}^{m}\lambda_{j}\sum_{i=1}^{n}\left[(\nabla f_{j}(\overline{x}))-A_{j}\overline{x})_{i}(x_{i}-\overline{x}_{i})+\frac{\chi_{i}(\overline{x})(\nabla f_{j}(\overline{x})-A_{j}\overline{x})_{i}}{(v_{i}-u_{i})}(x_{i}-\overline{x}_{i})^{2}\right]\geq0.$$
	Thereby $$f_{0}(x)-f_{0}(\overline{x})\geq L(x,\lambda)-L(\overline{x},\lambda)\geq0,$$
	for all $x \in\tilde{D}$. Hence $\overline{x}\in\tilde{D}$ is a global minimizer of ($P_2$). 
\end{proof}
The next theorem provide another sufficient condition in terms of given data, for a local minimizer to be a global minimizer of ($P_2$).
\begin{theorem}\label{tsc2}
	Let $\overline{x}\in\tilde{D}$ be a local minimizer of ($P_2$) and $\lambda_{0}=1$. Suppose that $\lambda_{j}\in\mathbb{R}^{+};~j\in \{1,2,\cdots,m\}$ are the Lagrangian multipliers associated with $\overline{x}\in\tilde{D}$ as given in the lemma (\ref{lem1}). If
	\begin{equation}\label{SC2RC}
	\sum_{j=0}^{m}\lambda_{j}\left[A_{j}+\mathrm{diag}\left(2\chi_{1}(\overline{x})\frac{(\nabla f_{j}(\overline{x})-A_{j}\overline{x})_{1}}{(v_{1}-u_{1})},
	\cdots,2\chi_{n}(\overline{x})\frac{(\nabla f_{j}(\overline{x})-A_{j}\overline{x})_{n}}{(v_{n}-u_{n})}\right)\right]\preceq0,
	\end{equation}
	then $\overline{x}$ is a global minimizer of ($P_2$).
\end{theorem}
\begin{proof}
	Let $\overline{x}\in\tilde{D}$ be a local minimizer and take, for each $i\in \{1,2,3,\cdots,n\}$ and $j\in \{0,1,2,\cdots,m\}$,
	$$q_{i}^{(j)}=2\chi_{i}(\overline{x})\frac{(\nabla f_{j}(\overline{x})-A_{j}\overline{x})_{i}}{(v_{i}-u_{i})}.$$
	Now from the Kuhn-Tucker conditions for (MP), we have there exists $\lambda_{j}\in\mathbb{R}^{+};~j\in \{0,1,2,\cdots,m\}$ with $\lambda_{0}=1$ such that
	$f_{0}(\overline{x})= L(\overline{x},\lambda)$ and $$f_{0}(x)\geq L(x,\lambda),~~\forall\,x \in\tilde{D}.$$
	So we have
	\begin{equation}\label{eq1RC2}
	f_{0}(x)-f_{0}(\overline{x})\geq L(x,\lambda)-L(\overline{x},\lambda),~~\forall\,x \in\tilde{D}.
	\end{equation}
	Now for any $x \in\tilde{D}$,
	\begin{eqnarray*}
		L(x,\lambda)-L(\overline{x},\lambda)&=&\sum_{j=0}^{m}\lambda_{j}\left[\left(f_j(x)-\frac{1}{2}x^TA_jx\right)-\left(f_j(\overline{x})-\frac{1}{2}\overline{x}^TA_j\overline{x}\right)\right]\\
		&=& \sum_{j=0}^{m}\lambda_{j}\left(f_j(x)-f_j(\overline{x})\right)- \sum_{j=0}^{m}\frac{1}{2}\lambda_{j}\left(x^TA_jx-\overline{x}^TA_j\overline{x}\right)\\
		&=& \sum_{j=0}^{m}\lambda_{j}\left(f_j(x)-f_j(\overline{x})\right)-\sum_{j=0}^{m}\lambda_{j}\left[\frac{1}{2}(x-\overline{x})^{T}A_{j}(x-\overline{x})+(A_{j}\overline{x})^{T}(x-\overline{x})\right].\\		
	\end{eqnarray*}
Thus, \begin{eqnarray*}
	L(x,\lambda)-L(\overline{x},\lambda)&\geq&\sum_{j=0}^{m}-\frac{1}{2}\lambda_{j}\left[(x-\overline{x})^{T}A_{j}(x-\overline{x})+(\nabla f_{j}(\overline{x})-A_{j}\overline{x})^{T}(x-\overline{x})\right]\\&=&\sum_{j=0}^{m}-\frac{1}{2}\lambda_{j}(x-\overline{x})^{T}[A_{j}+
	\mathrm{diag}(q_{1}^{(j)},q_{2}^{(j)},\cdots,q_{n}^{(j)})](x-\overline{x})\\
	&&+\sum_{j=0}^{m}\frac{1}{2}\lambda_{j}(x-\overline{x})^{T}\mathrm{diag}
	(q_{1}^{(j)},q_{2}^{(j)},\cdots,q_{n}^{(j)})(x-\overline{x})\\
	&&+\sum_{j=0}^{m}\lambda_{j}(\nabla f_{j}(\overline{x})-A_{j}\overline{x})^{T}(x-\overline{x})\\
	&=&\sum_{j=0}^{m}-\frac{1}{2}\lambda_{j}(x-\overline{x})^{T}[A_{j}+
	\mathrm{diag}(q_{1}^{(j)},q_{2}^{(j)},\cdots,q_{n}^{(j)})](x-\overline{x})\\
	&&+\sum_{j=0}^{m}\lambda_{j}\left[\sum_{i=1}^{n}\frac{1}{2}q_{i}^{(j)}(x_{i}-\overline{x}_{i})^{2}+(\nabla f_{j}(\overline{x})-A_{j}\overline{x})_{i}(x_{i}-\overline{x}_{i})\right].\\
\end{eqnarray*}
	From the condition (\ref{SC2RC}), we have $[A_{j}+\mathrm{diag}(q_{1}^{(j)},q_{2}^{(j)},\cdots,q_{n}^{(j)})]\preceq0$.
	Thus, for all $x \in\tilde{D}$, $$\sum_{j=0}^{m}-\frac{1}{2}\lambda_{j}(x-\overline{x})^{T}[A_{j}+
	\mathrm{diag}(q_{1}^{(j)},q_{2}^{(j)},\cdots,q_{n}^{(j)})](x-\overline{x})\geq0.$$
	Therefore $$L(x,\lambda)-L(\overline{x},\lambda)\geq\sum_{j=0}^{m}\lambda_{j}\left[\sum_{i=1}^{n}\frac{1}{2}q_{i}^{(j)}(x_{i}-\overline{x}_{i})^{2}+(\nabla f_{j}(\overline{x})-A_{j}\overline{x})_{i}(x_{i}-\overline{x}_{i})\right],$$ for all $x \in\tilde{D}$. Now our claim is 
	$$\sum_{j=0}^{m}\lambda_{j}\left[\sum_{i=1}^{n}\frac{1}{2}q_{i}^{(j)}(x_{i}-\overline{x}_{i})^{2}+(\nabla f_{j}(\overline{x})-A_{j}\overline{x})_{i}(x_{i}-\overline{x}_{i})\right]\geq0,$$ for all $x \in\tilde{D}$, and it follows from the necessary local optimality condition (\ref{lcl2RC}). \\For let $x=(x_{i})\in\tilde{D}$ arbitrarily.\\
	\textbf{Case-1:} If $\overline{x}_{i}=u_{i}$, then we have $\displaystyle\sum_{j=0}^{m}\lambda_{j}(\nabla f_{j}(\overline{x})-A_{j}\overline{x})_{i}\geq0$.
	Now 
	\begin{eqnarray*}
		~&&\sum_{j=0}^{m}\lambda_{j}\left[\frac{1}{2}q_{i}^{(j)}(x_{i}-\overline{x}_{i})^{2}+(\nabla f_{j}(\overline{x})-A_{j}\overline{x})_{i}(x_{i}-\overline{x}_{i})\right]\\
		&=&\sum_{j=0}^{m}\lambda_{j}\left[(-1)\frac{(\nabla f_{j}(\overline{x})-A_{j}\overline{x})_{i}}{(v_{i}-u_{i})}(x_{i}-\overline{x}_{i})^{2}+(\nabla f_{j}(\overline{x})-A_{j}\overline{x})_{i}(x_{i}-\overline{x}_{i})\right]\\
		&=&(x_{i}-u_{i})\left[\sum_{j=0}^{m}\lambda_{j}\frac{(\nabla f_{j}(\overline{x})-A_{j}\overline{x})_{i}}{(v_{i}-u_{i})}\right]\left(1-\frac{(x_{i}-u_{i})}{(v_{i}-u_{i})}\right).
	\end{eqnarray*}
	Therefore 
	$$\sum_{j=0}^{m}\lambda_{j}\left[\frac{1}{2}q_{i}^{(j)}(x_{i}-\overline{x}_{i})^{2}+(\nabla f_{j}(\overline{x})-A_{j}\overline{x})_{i}(x_{i}-\overline{x}_{i})\right]$$
	is non-negative, if $i\in I$ and is equal to $0$, if $i\in J$.\\
	\textbf{Case-2:} If $\overline{x}_{i}=v_{i}$, then $\displaystyle\sum_{j=0}^{m}\lambda_{j}(\nabla f_{j}(\overline{x})-A_{j}\overline{x})_{i}\leq0$ and in a similar method of Case-1, we will have
	$$\sum_{j=0}^{m}\lambda_{j}\left[\frac{1}{2}q_{i}^{(j)}(x_{i}-\overline{x}_{i})^{2}+(\nabla f_{j}(\overline{x})-A_{j}\overline{x})_{i}(x_{i}-\overline{x}_{i})\right]$$
	is non-negative, if $i\in I$ and is equal to $0$, if $i\in J$.\\ 
	\textbf{Case-3:} If $\overline{x}_{i}\in(u_{i},v_{i})$ and $i\in I$, then $\displaystyle\sum_{j=0}^{m}\lambda_{j}(\nabla f_{j}(\overline{x})-A_{j}\overline{x})_{i}=0$
	and so $\displaystyle\sum_{j=0}^{m}\lambda_{j}q_{i}^{(j)}=0$. Thus
	$$\sum_{j=0}^{m}\lambda_{j}\left[\frac{1}{2}q_{i}^{(j)}(x_{i}-\overline{x}_{i})^{2}+(\nabla f_{j}(\overline{x})-A_{j}\overline{x})_{i}(x_{i}-\overline{x}_{i})\right]=0.$$
	The summary of the above all three cases gives us that
	$$\sum_{j=0}^{m}\lambda_{j}\left[\sum_{i=1}^{n}\frac{1}{2}q_{i}^{(j)}(x_{i}-\overline{x}_{i})^{2}+(\nabla f_{j}(\overline{x})-A_{j}\overline{x})_{i}(x_{i}-\overline{x}_{i})\right]\geq0.$$
	Thereby $$f_{0}(x)-f_{0}(\overline{x})\geq L(x,\lambda)-L(\overline{x},\lambda)\geq0,$$
	for all $x \in\tilde{D}$. Hence $\overline{x}\in\tilde{D}$ is a global minimizer of ($P_2$).
\end{proof}
\begin{example}
	Consider the following quadratic weakly convex minimization problem ($E_3$):
	\begin{eqnarray*}
		\mbox{($E_3$)~~}~~\min_{(x_1,x_2)^\intercal\in\mathbb{R}^{2}} g_0\begin{pmatrix}
			x_1 \\ x_2
		\end{pmatrix}&=&\min_{(x_1,x_2)^\intercal\in\mathbb{R}^{2}} (1+x_1)^4+x_1^2+x_2^2-4x_1x_2+6x_1+6x_2 \\
		\mbox{s.t.~~}~~ g_1\begin{pmatrix}
			x_1 \\ x_2
		\end{pmatrix}&=&(1+x_2)^4+x_1^2+x_2^2-2\leq 0,\\
		&&x_1\in [-1,1],\\
		&&x_2\in \{-1,1\}.
	\end{eqnarray*}
	We can write $g_0\begin{pmatrix}
	x_1 \\ x_2
	\end{pmatrix}$ and $g_1\begin{pmatrix}
	x_1 \\ x_2
	\end{pmatrix}$ in the form 
	$$g_0(\mathbf{x})=f_0(\mathbf{x})-\dfrac{1}{2}\mathbf{x}^\intercal A_0\mathbf{x}\mbox{~~and~~}g_1(\mathbf{x})=f_1(\mathbf{x})-\dfrac{1}{2}\mathbf{x}^\intercal A_1\mathbf{x}$$  respectively, where
	$\mathbf{x}=\begin{pmatrix}
	x_1 \\ x_2
	\end{pmatrix}$,  $A_0=\begin{bmatrix}
	-2 & ~\,\,4\\ ~\,\,4 & -2
	\end{bmatrix}$, $A_1=\begin{bmatrix}
	-2 & ~\,\,0\\ ~\,\,0 & -2
	\end{bmatrix}$, and $f_0(\mathbf{x})=(1+x_1)^4+6x_1+6x_2$, $f_1(\mathbf{x})=(1+x_2)^4-2$ are convex functions. Let $u_1=u_2=-1$ and $v_1=v_2=1$.
	The Lagrangian of ($E_3$) is 
	\begin{equation*}\label{lagE3}
	L(\mathbf{x},\lambda,\mu,\nu)=[(1+x_1)^4+x_1^2+x_2^2-4x_1x_2+6x_1+6x_2]+\lambda((1+x_2)^4+x_1^2+x_2^2-2)+\mu(x_1^2-1)+\nu(x_2^2-1).
	\end{equation*}
	Apply the KKT conditions for ($E_3$). Then, we get
	\begin{itemize}
		\item[(1)] $\dfrac{\partial L(\mathbf{x},\lambda,\mu,\nu)}{\partial x_1}=[4(1+x_1)^3+2x_1-4x_2+6]+2\lambda x_1+2\mu x_1=0$,
		\item[(2)] $\lambda,\mu\geq0$, $(1+x_2)^4+x_1^2+x_2^2-2\leq 0$, $(x_1^2-1)\leq0$,
		\item[(3)] $\lambda((1+x_2)^4+x_1^2+x_2^2-2)=0$,
		\item[(4)] $\mu(x_1^2-1)=0$,
		\item[(5)] $\nu=0$ and $x_2=-1$ or $x_2=1$.
	\end{itemize}
There are infinitely many solutions $\{(-1,-1,\lambda,\mu,0)~\mid~\lambda+\mu=4~~\mbox{and}~~\lambda,\mu\geq0\}$ of KKT  for the problem ($E_3$). Notice that the linear independence constraint qualification is satisfied in this problem,  for $\mathbf{x}=(-1,-1)^\intercal$. Direct calculations show that the first eigenvalues of $-\frac{1}{2}A_0$ and $-\frac{1}{2}A_1$ are
	$\mu_0=3\mbox{~~and~~}\mu_1=1$ 
	respectively. One can directly see that the global optimality conditions (\ref{SC1RC}) and (\ref{SC2RC}) are satisfied for all $\lambda\in[0,4]$. Therefore, $\mathbf{x}=(-1,-1)^\intercal$ is the global minimizer.
\end{example}
\begin{remark}
Let $\overline{x}\in\tilde{D}$ be a local minimizer of ($P_2$) and $\lambda_{0}=1$. Suppose that $\lambda_{j}\in\mathbb{R}^{+};~j\in \{1,2,\cdots,m\}$ are the Lagrangian multipliers associated with $\overline{x}\in\tilde{D}$ as given in the lemma (\ref{lem1}) and $\mu_{j}\in\R$ are the first eigenvalues of the symmetric matrices $A_{j}$, for all $j\in\{0,1,2,\cdots,m\}$, then the condition (\ref{SC}) implies (\ref{SC1RC}) and the condition (\ref{SC1RC}) implies (\ref{SC2RC}). That is,
\begin{equation*}
\sum_{j=0}^m\lambda_j\mu_j\geq-\inf_{x\neq\overline{x}}\frac{(\nabla L(\overline{x},\lambda))^{T}(x-\overline{x})}{\|x-\overline{x}\|^{2}}\Longrightarrow\sum_{j=0}^{m}\lambda_{j}\sum_{i=1}^{n}\left[\frac{\chi_{i}(\overline{x})(\nabla f_{j}(\overline{x})-A_{j}\overline{x})_{i}}{(v_{i}-u_{i})}\right]\leq\sum_{j=0}^{m}\lambda_{j}\mu_{j}
\end{equation*}
and
\begin{eqnarray*}
	&&\sum_{j=0}^{m}\lambda_{j}\sum_{i=1}^{n}\left[\frac{\chi_{i}(\overline{x})(\nabla f_{j}(\overline{x})-A_{j}\overline{x})_{i}}{(v_{i}-u_{i})}\right]\leq\sum_{j=0}^{m}\lambda_{j}\mu_{j}\\
	&\Longrightarrow&\sum_{j=0}^{m}\lambda_{j}\left[A_{j}+\mathrm{diag}\left(2\chi_{1}(\overline{x})\frac{(\nabla f_{j}(\overline{x})-A_{j}\overline{x})_{1}}{(v_{1}-u_{1})},
	\cdots,2\chi_{n}(\overline{x})\frac{(\nabla f_{j}(\overline{x})-A_{j}\overline{x})_{n}}{(v_{n}-u_{n})}\right)\right]\preceq0
\end{eqnarray*}
For,let $\overline{x}\in\tilde{D}$ be a local minimizer of ($P_2$). Suppose that
$$\sum_{j=0}^m\lambda_j\mu_j\geq-\inf_{x\neq\overline{x}}\frac{(\nabla L(\overline{x},\lambda))^{T}(x-\overline{x})}{\|x-\overline{x}\|^{2}}.$$
Then 
\begin{equation}\label{eq1RC3}
\sum_{j=0}^m\lambda_j\mu_j\geq-\frac{\sum_{i=1}^{n}(\nabla L(\overline{x},\lambda))_{i}
	(x-\overline{x})_{i}}{\sum_{i=1}^{n}(x_{i}-\overline{x}_{i})^{2}},~~\forall\,x=(x_{i})\in\tilde{D}~~\mbox{~with~}~~x\neq\overline{x}.
\end{equation}
Now let $i\in \{1,2,\cdots,n\}$ and $t\in[u_{i},v_{i}]$ with $t\neq\overline{x}_{i}$ and define the vector $x=(x_{k})\in\mathbb{R}^{n}$ such that
$$x_{k}:=\left\{
\begin{array}{ll}
t & \mathrm{if}\ k=i \\
\overline{x}_{k} & \mathrm{if}\ k\neq i.\\
\end{array}%
\right. $$
Then clearly $x\neq\overline{x}$ and (\ref{eq1RC3}) splits into
$$\sum_{j=0}^m\lambda_j\mu_j\geq-\frac{(\nabla L(\overline{x},\lambda))_{i}
}{(t-\overline{x}_{i})}.$$
That is, for each $i\in \{1,2,\cdots,n\}$,
\begin{equation}\label{eq1RC4}
\sum_{j=0}^m\lambda_j\mu_j\geq-\frac{(\nabla L(\overline{x},\lambda))_{i}
}{(t-\overline{x}_{i})},~~\forall\,t\in[u_{i},v_{i}]~~\mbox{~with~}~~t\neq\overline{x}_{i}.
\end{equation}
\textbf{Case-1:} If $\overline{x}_{i}=u_{i}$, then choosing $t=v_{i}$ and by the local optimality condition (\ref{lcl2RC}), we have 
$$\displaystyle (\nabla L(\overline{x},\lambda))_{i}=\sum_{j=0}^{m}\lambda_{j}(\nabla f_{j}(\overline{x})-A_{j}\overline{x})_{i}\geq0$$
as $\chi_{i}(\overline{x})=-1$ and 
$$-\frac{(\nabla L(\overline{x},\lambda))_{i}}{(v_{i}-u_{i})}=\sum_{j=0}^{m}\lambda_{j}\frac{\chi_{i}(\overline{x})(\nabla f_{j}(\overline{x})
	-A_{j}\overline{x})_{i}}{(v_{i}-u_{i})}\leq0,~~\forall\,i\in \{1,2,\cdots,n\}.$$
Thus $$\sum_{j=0}^m\lambda_j\mu_j\geq\sum_{j=0}^{m}\lambda_{j}\frac{\chi_{i}(\overline{x})(\nabla f_{j}(\overline{x})
	-A_{j}\overline{x})_{i}}{(v_{i}-u_{i})}\geq\sum_{j=0}^{m}\lambda_{j}\sum_{i=1}^{n}\left[\frac{\chi_{i}(\overline{x})(\nabla f_{j}(\overline{x})-A_{j}\overline{x})_{i}}{(v_{i}-u_{i})}\right].$$
\textbf{Case-2:} If $\overline{x}_{i}=v_{i}$, then choosing $t=u_{i}$ and by the local optimality condition (\ref{lcl2RC}), we have 
$$\displaystyle (\nabla L(\overline{x},\lambda))_{i}=\sum_{j=0}^{m}\lambda_{j}(\nabla f_{j}(\overline{x})-A_{j}\overline{x})_{i}\leq0$$
as $\chi_{i}(\overline{x})=1$ and 
$$-\frac{(\nabla L(\overline{x},\lambda))_{i}}{(v_{i}-u_{i})}=\sum_{j=0}^{m}\lambda_{j}\frac{\chi_{i}(\overline{x})(\nabla f_{j}(\overline{x})
	-A_{j}\overline{x})_{i}}{(v_{i}-u_{i})}\leq0,~~\forall\,i\in \{1,2,\cdots,n\}.$$
Thus $$\sum_{j=0}^m\lambda_j\mu_j\geq\sum_{j=0}^{m}\lambda_{j}\frac{\chi_{i}(\overline{x})(\nabla f_{j}(\overline{x})
	-A_{j}\overline{x})_{i}}{(v_{i}-u_{i})}\geq\sum_{j=0}^{m}\lambda_{j}\sum_{i=1}^{n}\left[\frac{\chi_{i}(\overline{x})(\nabla f_{j}(\overline{x})-A_{j}\overline{x})_{i}}{(v_{i}-u_{i})}\right].$$
\textbf{Case-3:} If $\overline{x}_{i}\in(u_{i},v_{i})$ and $i\in I$, then by the local optimality condition (\ref{lcl2RC}), we have 
$$\left(\sum_{j=0}^{m}\lambda_{j}(\nabla f_{j}(\overline{x})-A_{j}\overline{x})_{i}\right)^{2}\leq0$$ as $\chi_{i}(\overline{x})=\displaystyle\sum_{j=0}^{m}\lambda_{j}(\nabla f_{j}(\overline{x})-A_{j}\overline{x})_{i}$. That is, $$\nabla L(\overline{x},\lambda)=\left(\sum_{j=0}^{m}\lambda_{j}(\nabla f_{j}(\overline{x})-A_{j}\overline{x})_{i}\right)=0.$$
Thus (\ref{eq1RC4}) implies $$\sum_{j=0}^m\lambda_j\mu_j\geq0=\sum_{j=0}^{m}\lambda_{j}\frac{\chi_{i}(\overline{x})(\nabla f_{j}(\overline{x})
	-A_{j}\overline{x})_{i}}{(v_{i}-u_{i})}=\sum_{j=0}^{m}\lambda_{j}\sum_{i=1}^{n}\left[\frac{\chi_{i}(\overline{x})(\nabla f_{j}(\overline{x})-A_{j}\overline{x})_{i}}{(v_{i}-u_{i})}\right].$$
Summing up all three above cases, the first implication follows.\\
Now we see the second implication, for suppose that
\begin{equation*}\label{eq1RC5}
\sum_{j=0}^{m}\lambda_{j}\sum_{i=1}^{n}\left[\frac{\chi_{i}(\overline{x})(\nabla f_{j}(\overline{x})-A_{j}\overline{x})_{i}}{(v_{i}-u_{i})}\right]\leq\sum_{j=0}^{m}\lambda_{j}\mu_{j}.
\end{equation*}
Let $x\in\R^{n}$, then
\begin{eqnarray*}
	&&x^{T}\left(\sum_{j=0}^{m}\lambda_{j}\left[A_{j}+\mathrm{diag}\left(2\chi_{1}(\overline{x})\frac{(\nabla f_{j}(\overline{x})-A_{j}\overline{x})_{1}}{(v_{1}-u_{1})},
	\cdots,2\chi_{n}(\overline{x})\frac{(\nabla f_{j}(\overline{x})-A_{j}\overline{x})_{n}}{(v_{n}-u_{n})}\right)\right]\right)x\\
	&=&\sum_{j=0}^{m}\lambda_{j}\left[x^{T}A_{j}x+x^{T}\mathrm{diag}\left(2\chi_{1}(\overline{x})\frac{(\nabla f_{j}(\overline{x})-A_{j}\overline{x})_{1}}{(v_{1}-u_{1})},
	\cdots,2\chi_{n}(\overline{x})\frac{(\nabla f_{j}(\overline{x})-A_{j}\overline{x})_{n}}{(v_{n}-u_{n})}\right)x\right]\\
	&\leq&\sum_{j=0}^{m}\lambda_{j}\left[-2\mu_{j}\|x\|^{2}+x^{T}\mathrm{diag}\left(2\chi_{1}(\overline{x})\frac{(\nabla f_{j}(\overline{x})-A_{j}\overline{x})_{1}}{(v_{1}-u_{1})},
	\cdots,2\chi_{n}(\overline{x})\frac{(\nabla f_{j}(\overline{x})-A_{j}\overline{x})_{n}}{(v_{n}-u_{n})}\right)x\right]\\
	&=&-2\sum_{j=0}^{m}\lambda_{j}\sum_{i=1}^{n}\left[\mu_{j}-\frac{\chi_{i}(\overline{x})(\nabla f_{j}(\overline{x})-A_{j}\overline{x})_{i}}{(v_{i}-u_{i})}\right]x_{i}^{2}\\
	&=&-2\sum_{j=0}^{m}\lambda_{j}\left[\mu_{j}-\sum_{i=1}^{n}\frac{\chi_{i}(\overline{x})(\nabla f_{j}(\overline{x})-A_{j}\overline{x})_{i}}{(v_{i}-u_{i})}\right]x_{i}^{2}\\
	&\leq&0.
\end{eqnarray*}
Therefore the second implication holds. Hence the remark follows.
\end{remark}
\section{Quadratic Fractional Programming Problem}
In this section, results of  problem ($P_1$) have been extended for the following quadratic fractional programming model problem with both continuous and discrete variables:
\begin{eqnarray*}
	\mbox{($P_3$)~~}~~\min_{x\in\mathbb{R}^{n}}s(x)=\min_{x\in\mathbb{R}^{n}} \frac{f_{0}(x)}{g_{0}(x)}&=&\min_{x\in\mathbb{R}^{n}} \frac{\frac{1}{2}x^{T}A_{0}x+a_{0}^{T}x+c_{0}}{\frac{1}{2}x^{T}B_{0}x+b_{0}^{T}x+d_{0}}\\
	\mbox{subject to~~}~~\frac{f_{j}(x)}{g_{j}(x)} &=&\frac{\frac{1}{2}x^{T}A_{j}x+a_{j}^{T}x+c_{j}}{\frac{1}{2}x^{T}B_{j}x+b_{j}^{T}x+d_{j}}\leq e_j,~~\forall~j\in\{1,2,3,\cdots,m\}\\
	&&x_{i} \in [u_{i},v_{i}], ~i\in I \\
	&&x_{i} \in \{u_{i},v_{i}\}, ~i\in J ;
\end{eqnarray*}
where $I\cap J=\phi$, $I\cup J=\{1,2,3,\cdots,n\}$, for each $j\in\{0,1,2,3,\cdots,m\}$,  $A_{j}=(a_{st}^{(j)})$, $B_{j}=(b_{st}^{(j)})$ are $n\times n$ symmetric matrices, $a_{j}=(a_{r}^{(j)}),b_{j}=(b_{r}^{(j)})\in\mathbb{R}^{n}$, $c_{j},d_{j}, e_j\in\mathbb{R}$, 
and $u_{i},v_{i}\in\mathbb{R}$ with $u_{i}<v_{i},$ for all $i\in \{1,2,3,\cdots,n\}$.\\
Now for $\lambda=(\lambda_{j})\in\mathbb{R}^{m}_{+}$, define the Lagrangian $L(\cdot,\lambda)$ of ($P_3$) by
\begin{equation}\label{lag_MPQF}
L(x,\lambda):=\frac{f_{0}(x)}{g_{0}(x)}+\sum_{j=1}^{m}\lambda_{j}\left( \frac{f_{j}(x)}{g_{j}(x)}-e_j\right) ;~~x\in\mathbb{R}^{n}.
\end{equation}
That is, \begin{equation}
\label{lagQF}L(x,\lambda)=\frac{\frac{1}{2}x^{T}A_{0}x+a_{0}^{T}x+c_{0}}{\frac{1}{2}x^{T}B_{0}x+b_{0}^{T}x+d_{0}}+\sum_{j=1}^{m}\lambda_{j}\left(\frac{\frac{1}{2}x^{T}A_{j}x+a_{j}^{T}x+c_{j}}{\frac{1}{2}x^{T}B_{j}x+b_{j}^{T}x+d_{j}}-e_j\right).
\end{equation} 
For $\overline{x}\in D$ and $i\in \{1,2,3,\cdots,n\}$, define
\begin{equation}\label{eq1QF}
\tilde{\chi}_{i}(\overline{x}):=\left\{
\begin{array}{lll}
-1 & \mathrm{if}\ \overline{x}_i=u_i \\
+1 & \mathrm{if}\ \overline{x}_i=v_i\\
c\nabla L(\overline{x},\lambda)_{i} & \mathrm{if}\ \overline{x}_i\in (u_i ,v_i );
\end{array}%
\right. 
\end{equation}
where $c=\frac{1}{2}\overline{x}^{T}B_{0}\overline{x}+b_{0}^{T}\overline{x}+d_{0}$. 
Assume for each $x\in D$, $g_j(x)>0$, if $j\in I_P\cup\{0\}$ and $g_j(x)<0$, if $j\in I_N$,
where $I_P\cap I_N=\phi$ and $I_P\cup I_N=\{1,2,3,\cdots m\}$. Let
\begin{equation*}
\xi_j:=\left\{
\begin{array}{lll}
+1 & \mathrm{if}\ j\in I_P\cup\{0\} \\
-1 & \mathrm{if}\ j\in I_N.
\end{array}%
\right. 
\end{equation*}	
Suppose that, $\overline{x}\in\tilde{D}$ is a local minimizer of ($P_3$). Now we formulate a quadratic programming problem ($P_3^\ast$) by preserving the local and global minimizers of the original fractional programming problem ($P_3$) as follows:
\begin{eqnarray*}
	\mbox{($P_3^\ast$)~~}~~\min_{x\in\mathbb{R}^{n}} \mathfrak{h}_{0}(x)&=&\min_{x\in\mathbb{R}^{n}} \frac{1}{2}x^{T}\mathfrak{Q}_0x+\mathfrak{q}_{0}^{T}x+\mathfrak{r}_{0} \\
	\mbox{s.t.~~}~~ \mathfrak{h}_{j}(x) &=&\frac{1}{2}x^{T}\mathfrak{Q}_{j}x+\mathfrak{q}_{j}^{T}x+\mathfrak{r}_{j}\leq 0,~~\forall~j\in\{1,2,3,\cdots,m\}\\
	&&x_{i} \in [u_{i},v_{i}], ~i\in I \\
	&&x_{i} \in \{u_{i},v_{i}\}, ~i\in J ;
\end{eqnarray*}
where for each $j=0,1,2,\ldots,m$,
$\mathfrak{Q}_{j}=\xi_j(A_j-e_jB_j)$, $\mathfrak{q}_{j}=\xi_j(a_j-e_jb_j)$, $\mathfrak{r}_{j}=\xi_j(c_j-e_jd_j)$, $e_0=s(\overline{x})$. One can easily see that $\mathfrak{h}_{0}(\overline{x})=0$ and for each $x\in D$, 
\begin{equation}\label{lomin}
\mathfrak{h}_{0}(x)-\mathfrak{h}_{0}(\overline{x})=g_0(x)(s(x)-s(\overline{x})).
\end{equation}
Hence $\overline{x}$ is a local minimizer of ($P_3$) if and only if $\overline{x}$ is a local minimizer of ($P_3^\ast$). Since $g_0(x)>0$, for all $x\in D$, we have from (\ref{lomin}), $\overline{x}$ is a global minimizer of ($P_3$) if and only if $\overline{x}$ is a global minimizer of ($P_3^\ast$).\\
The Lagrangian of ($P_3$) can be written as $$\mathfrak{L}(x,\lambda)=\mathfrak{h}_{0}(x)+\sum_{j=1}^{m}\lambda_{j}\mathfrak{h}_{j}(x).$$ 
\begin{lemma}\label{LM}
If $\lambda=(\lambda_{j})\in\mathbb{R}^{m}_{+}$ is the Lagrange multiplier of ($P_3$) associated with $\overline{x}\in \tilde{D}$, then  $\mu=(\mu_{j})$ is the Lagrange multiplier of ($P_3^\ast$), where $\mu_j=c\lambda_j(\frac{1}{2}\overline{x}^{T}B_{j}\overline{x}+b_{j}^{T}\overline{x}+d_{j})^{-1}$, for all $j=1,2,3,\cdots,m$ and $c=\frac{1}{2}\overline{x}^{T}B_{0}\overline{x}+b_{0}^{T}\overline{x}+d_{0}$.
\end{lemma}
\begin{proof}
Suppose that $\lambda=(\lambda_{j})\in\mathbb{R}^{m}_{+}$ is the Lagrange multiplier of ($P_3$) associated with $\overline{x}\in \tilde{D}$. To show $\lambda=(\lambda_{j})$ is the Lagrange multiplier of ($P_3^\ast$) associated with $\overline{x}\in \tilde{D}$, it suffices to obtain that the following two conditions are satisfied:
\begin{equation}\label{mfkkt1}
\mu_j\left( \frac{1}{2}\overline{x}^{T}\mathfrak{Q}_{j}\overline{x}+\mathfrak{q}_{j}^{T}\overline{x}+\mathfrak{r}_{j}\right)=0,~~\mbox{for all}~~j=1,2,3,\cdots,m\quad\text{and}
\end{equation}
\begin{equation}\label{mfkkt2}
\nabla \mathfrak{L}(\overline{x},\mu)=(\mathfrak{Q}_{0}\overline{x}+\mathfrak{q}_{0})+\sum_{j=1}^{m}\mu_j(\mathfrak{Q}_{j}\overline{x}+\mathfrak{q}_{j})=0.
\end{equation}
Since $\lambda=(\lambda_{j})$ is the Lagrange multiplier of ($P_3$) associated with $\overline{x}\in \tilde{D}$, we have
\begin{equation}\label{fkkt1}
\lambda_j\left(\frac{\frac{1}{2}\overline{x}^{T}A_{j}\overline{x}+a_{j}^{T}\overline{x}+c_{j}}{\frac{1}{2}\overline{x}^{T}B_{j}\overline{x}+b_{j}^{T}\overline{x}+d_{j}}-e_j\right)=0,~~\mbox{for all}~~j=1,2,3,\cdots,m 
\end{equation}
and
\begin{equation}\label{fkkt2}
\nabla L(\overline{x},\lambda)=\nabla\left( \frac{\frac{1}{2}\overline{x}^{T}A_{0}\overline{x}+a_{0}^{T}\overline{x}+c_{0}}{\frac{1}{2}\overline{x}^{T}B_{0}\overline{x}+b_{0}^{T}\overline{x}+d_{0}}\right) + \sum_{j=1}^{m}\lambda_{j}\nabla\left(\frac{\frac{1}{2}\overline{x}^{T}A_{j}\overline{x}+a_{j}^{T}\overline{x}+c_{j}}{\frac{1}{2}\overline{x}^{T}B_{j}\overline{x}+b_{j}^{T}\overline{x}+d_{j}}-e_j\right) =0.  
\end{equation}
The condition (\ref{mfkkt1}) trivially follows from (\ref{fkkt1}). One can easily see that $$\mu_j\left(\frac{\frac{1}{2}\overline{x}^{T}A_{j}\overline{x}+a_{j}^{T}\overline{x}+c_{j}}{\frac{1}{2}\overline{x}^{T}B_{j}\overline{x}+b_{j}^{T}\overline{x}+d_{j}}-e_j\right)=0,~~\mbox{for all}~~j=1,2,3,\cdots,m. $$ Then for any $i=1,2,3,\cdots,n$,
\begin{eqnarray*}
	\nabla \mathfrak{L}(\overline{x},\mu)_i
	&=&(\mathfrak{Q}_{0}\overline{x}+\mathfrak{q}_{0})_i+\sum_{j=1}^{m}\mu_j(\mathfrak{Q}_{j}\overline{x}+\mathfrak{q}_{j})_i\\
	&=& (\mathfrak{Q}_{0}\overline{x}+\mathfrak{q}_{0})_i+\sum_{j=1}^{m}\mu_j\left[ (\mathfrak{Q}_{j}\overline{x}+\mathfrak{q}_{j})+\xi_j\left(\frac{\frac{1}{2}\overline{x}^{T}A_{j}\overline{x}+a_{j}^{T}\overline{x}+c_{j}}{\frac{1}{2}\overline{x}^{T}B_{j}\overline{x}+b_{j}^{T}\overline{x}+d_{j}}-e_j\right)(B_{j}\overline{x}+b_{j})\right]_i \\
	&=& (\mathfrak{Q}_{0}\overline{x}+\mathfrak{q}_{0})_i+\sum_{j=1}^{m}\mu_j\left[ (\mathfrak{Q}_{j}\overline{x}+\mathfrak{q}_{j})+\left(\frac{\frac{1}{2}\overline{x}^{T}\mathfrak{Q}_{j}\overline{x}+\mathfrak{q}_{j}^{T}\overline{x}+\mathfrak{r}_{j}}{\frac{1}{2}\overline{x}^{T}B_{j}\overline{x}+b_{j}^{T}\overline{x}+d_{j}}\right)(B_{j}\overline{x}+b_{j})\right]_i \\
	&=&(A_0\overline{x}+a_0)_i+s(\overline{x})(B_0\overline{x}+b_0)_i \sum_{j=1}^{m}\mu_j(\frac{1}{2}\overline{x}^{T}B_{j}\overline{x}+b_{j}^{T}\overline{x}+d_{j})\nabla\left[\frac{\frac{1}{2}\overline{x}^{T}\mathfrak{Q}_{j}\overline{x}+\mathfrak{q}_{j}^{T}\overline{x}+\mathfrak{r}_{j}}{\frac{1}{2}\overline{x}^{T}B_{j}\overline{x}+b_{j}^{T}\overline{x}+d_{j}}\right]_i\\
	&=&c\nabla\left( \frac{\frac{1}{2}\overline{x}^{T}A_{0}\overline{x}+a_{0}^{T}\overline{x}+c_{0}}{\frac{1}{2}\overline{x}^{T}B_{0}\overline{x}+b_{0}^{T}\overline{x}+d_{0}}\right)_i + c\sum_{j=1}^{m}\lambda_j\nabla\left[\frac{\frac{1}{2}\overline{x}^{T}\mathfrak{Q}_{j}\overline{x}+\mathfrak{q}_{j}^{T}\overline{x}+\mathfrak{r}_{j}}{\frac{1}{2}\overline{x}^{T}B_{j}\overline{x}+b_{j}^{T}\overline{x}+d_{j}}\right]_i\\
	&=&c\nabla L(\overline{x},\lambda)_i=0.
\end{eqnarray*}
This concludes the result.
\end{proof}
\subsection{Local necessary optimality condition for ($P_3$)}
In this section, we are going to derive a local necessary optimality condition for ($P_3$) as follows:
\begin{theorem}\label{OPTCdn}
	Let $\overline{x}\in\tilde{D}$ be a local minimizer of ($P_3$). Then 
	\begin{equation}\label{lcl2}
	\tilde{\chi}_{i}(\overline{x})\left(\left[ \frac{(A_{0}\overline{x}+a_{0})_{i}-s(\overline{x})(B_{0}\overline{x}+b_{0})_{i}}{\frac{1}{2}\overline{x}^{T}B_{0}\overline{x}+b_{0}^{T}\overline{x}+d_{0}}\right]+ \sum_{j=1}^{m}\xi_j\lambda_{j}\left[ \frac{(A_{j}\overline{x}+a_{j})_{i}-e_j(B_{j}\overline{x}+b_{j})_{i}}{\frac{1}{2}\overline{x}^{T}B_{j}\overline{x}+b_{j}^{T}\overline{x}+d_{j}}\right]\right)\leq0;
	\end{equation}	
	where $\lambda_{j}\in\mathbb{R}^+; j=1,2,3,\cdots,m$ are the Lagrangian multipliers associated with $\overline{x}\in\tilde{D}$.
\end{theorem}
\begin{proof}
We shall show that, for each $i\in I$,
\begin{equation}\label{eq1QF1}
\sum_{j=0}^{m}\mu_{j}(\mathfrak{Q}_{j}\overline{x}+\mathfrak{q}_{j})_{i}(t-\overline{x}_{i})\geq0,~~\forall\,t\in[u_{i},v_{i}];
\end{equation}
where $\mu_0=1$. For: let $i\in I$ and $y\in[u_{i},v_{i}]$ and define the vector $x=(x_{k})\in\mathbb{R}^{n}$ such that
$$x_{k}:=\left\{
\begin{array}{ll}
t & \mathrm{if}\ k=i \\
\overline{x}_{k} & \mathrm{if}\ k\neq i.\\
\end{array}%
\right. $$
Then clearly $x\in\Delta=\{x=(x_{i})\in\mathbb{R}^{n}~|~x_{i}\in[u_{i},v_{i}], ~i\in I\}$. Hence by the Lemma (\ref{lem1}), we have
$$\sum_{j=0}^{m}\mu_{j}(\mathfrak{Q}_{j}\overline{x}+\mathfrak{q}_{j})_{i}(t-\overline{x}_{i})\geq0,~~\forall\,t\in[u_{i},v_{i}].$$
Let $i\in I$ arbitrarily. \\
\textbf{Case-1:} If $\overline{x}_{i}=u_{i}$, then choose $t=v_{i}$. From (\ref{eq1QF}), we have
$$\tilde{\chi}_{i}(\overline{x})\sum_{j=0}^{m}\mu_{j}(\mathfrak{Q}_{j}\overline{x}+\mathfrak{q}_{j})_{i}\leq0.$$
\textbf{Case-2:} If $\overline{x}_{i}=v_{i}$, then choose $t=u_{i}$. From (\ref{eq1QF}), we have
$$\tilde{\chi}_{i}(\overline{x})\sum_{j=0}^{m}\mu_{j}(\mathfrak{Q}_{j}\overline{x}+\mathfrak{q}_{j})_{i}\leq0.$$
\textbf{Case-3:} If $\overline{x}_i\in (u_{i} ,v_{i} )$, then on the one hand, choose $t=u_{i}$. From (\ref{eq1QF}), we have
\begin{equation}\label{eq2}
\sum_{j=0}^{m}\mu_{j}(\mathfrak{Q}_{j}\overline{x}+\mathfrak{q}_{j})_{i}\leq0.
\end{equation}
On the other hand, choose $t=v_{i}$, the from (\ref{eq1QF}), we have
\begin{equation}\label{eq3}
\sum_{j=0}^{m}\mu_{j}(\mathfrak{Q}_{j}\overline{x}+\mathfrak{q}_{j})_{i}\geq0.
\end{equation}
From (\ref{eq2}) and (\ref{eq3}), we have $$\tilde{\chi}_{i}(\overline{x})=\nabla L(\overline{x},\lambda)_i=0.$$
Thus $$\tilde{\chi}_{i}(\overline{x})\sum_{j=0}^{m}\mu_{j}(\mathfrak{Q}_{j}\overline{x}+\mathfrak{q}_{j})_{i}=0.$$
Hence by the summary of all above three cases, we have
$$\tilde{\chi}_{i}(\overline{x})\left(\sum_{j=0}^{m}\mu_{j}(\mathfrak{Q}_{j}\overline{x}+\mathfrak{q}_{j})_{i}\right)\leq0,~~\forall\,i\in I.$$
Therefore, since $c=\frac{1}{2}\overline{x}^{T}B_{0}\overline{x}+b_{0}^{T}\overline{x}+d_{0}>0$ and $\mu_j=c\lambda_j(\frac{1}{2}\overline{x}^{T}B_{j}\overline{x}+b_{j}^{T}\overline{x}+d_{j})^{-1}$, for all $j=1,2,3,\cdots,m$, it follows that, for each $i\in I$
$$	\tilde{\chi}_{i}(\overline{x})\left(\left[ \frac{(A_{0}\overline{x}+a_{0})_{i}-s(\overline{x})(B_{0}\overline{x}+b_{0})_{i}}{\frac{1}{2}\overline{x}^{T}B_{0}\overline{x}+b_{0}^{T}\overline{x}+d_{0}}\right]+ \sum_{j=1}^{m}\xi_j\lambda_{j}\left[ \frac{(A_{j}\overline{x}+a_{j})_{i}-e_j(B_{j}\overline{x}+b_{j})_{i}}{\frac{1}{2}\overline{x}^{T}B_{j}\overline{x}+b_{j}^{T}\overline{x}+d_{j}}\right]\right) \leq0.$$
\end{proof}
\subsection{Verifiable global sufficient optimality condition for ($P_3$)}
In this section, we derive a global sufficient optimality condition for ($P_3$) as follows:
\begin{theorem}
	Let $\overline{x}\in \tilde{D}$ is a local minimizer of ($P_3$) and $\lambda_{0}=1$. If  $\lambda_{j}\in\mathbb{R}^{+};~j\in \{1,2,\cdots,m\}$ are the Lagrangian multipliers associated with $\overline{x}$ and
	\begin{equation}\label{GSC}
	\sum_{j=0}^{m}\left\lvert\dfrac{\lambda_j}{\frac{1}{2}\overline{x}^{T}B_{j}\overline{x}+b_{j}^{T}\overline{x}+d_{j}}\right\rvert\{[A_j-\mathrm{diag}(N_1,\cdots,N_n)]-e_j[B_j-\mathrm{diag}(D_1,\cdots,D_n)]\}\succeq0,
	\end{equation}
	where $N_i=2\tilde{\chi}_i(\overline{x})\dfrac{(A_{j}\overline{x}+a_{j})_{i}}{(v_{i}-u_{i})}$ and $D_i=2\tilde{\chi}_i(\overline{x})\dfrac{(B_{j}\overline{x}+b_{j})_{i}}{(v_{i}-u_{i})}$, for all $i\in\{1,2,3,\cdots,n\}$, then $\overline{x}$ is a global minimizer of ($P_3$).
\end{theorem}
\begin{proof}
Since $A_{j}=(a_{st}^{(j)})$ and $B_{j}=(b_{st}^{(j)})$ are $n\times n$ symmetric matrices, we have $\mathfrak{Q}_{j}=\xi_j(A_j-e_jB_j)$ is also a  symmetric matrix. By applying Theorem (\ref{tscq}) to ($P_3^\ast$), we have \begin{equation}\label{SC1}
\sum_{j=0}^{m}\mu_{j}\left[\mathfrak{Q}_{j}-\mathrm{diag}\left(2\tilde{\chi}_{1}(\overline{x})\frac{(\mathfrak{Q}_{j}\overline{x}+\mathfrak{q}_{j})_{1}}{(v_{1}-u_{1})},
\cdots,2\tilde{\chi}_{n}(\overline{x})\frac{(\mathfrak{Q}_{j}\overline{x}+\mathfrak{q}_{j})_{n}}{(v_{n}-u_{n})}\right)\right]\succeq0
\end{equation}
is a sufficient condition for $\overline{x}$ to be a global minimizers of ($P_3^\ast$). Thus (\ref{SC1}) produces a sufficient optimality condition for ($P_3$) as well. We shall now reduce (\ref{SC1}) into the desired form (\ref{GSC}) by making the substitutions  $\mathfrak{Q}_{j}=\xi_j(A_j-e_jB_j)$  and $\mathfrak{q}_{j}=\xi_j(a_j-e_jb_j)$ into (\ref{SC1}). Since $c>0$, we get
$$\sum_{j=0}^{m}\lambda_j\xi_j\left(\frac{1}{2}\overline{x}^{T}B_{j}\overline{x}+b_{j}^{T}\overline{x}+d_{j}\right)
^{-1}\{[A_j-\mathrm{diag}(N_1,\cdots,N_n)]-e_j[B_j-\mathrm{diag}(D_1,\cdots,D_n)]\}\succeq0,$$
where $N_i=2\tilde{\chi}_i(\overline{x})\dfrac{(A_{j}\overline{x}+a_{j})_{i}}{(v_{i}-u_{i})}$ and $D_i=2\tilde{\chi}_i(\overline{x})\dfrac{(B_{j}\overline{x}+b_{j})_{i}}{(v_{i}-u_{i})}$, for all $i\in\{1,2,3,\cdots,n\}$. Since for each $j\in\{1,2,3,\cdots,m\}$, $$\xi_j\left(\frac{1}{2}\overline{x}^{T}B_{j}\overline{x}+b_{j}^{T}\overline{x}+d_{j}\right)
^{-1}= \left\lvert\dfrac{1}{\frac{1}{2}\overline{x}^{T}B_{j}\overline{x}+b_{j}^{T}\overline{x}+d_{j}}\right\rvert,$$ the global sufficient condition becomes
$$\sum_{j=0}^{m}\left\lvert\dfrac{\lambda_j}{\frac{1}{2}\overline{x}^{T}B_{j}\overline{x}+b_{j}^{T}\overline{x}+d_{j}}\right\rvert\{[A_j-\mathrm{diag}(N_1,\cdots,N_n)]-e_j[B_j-\mathrm{diag}(D_1,\cdots,D_n)]\}\succeq0.$$
Hence the Theorem follows.
\end{proof}
Let us illustrate the significance of the above optimality condition.
\begin{example}
	Consider the following quadratic non-convex minimization problem ($E_4$):
	\begin{eqnarray*}
		\mbox{($E_4$)~~}~~\min_{(x_1,x_2)^\intercal\in\mathbb{R}^{2}} f\begin{pmatrix}
			x_1 \\ x_2
		\end{pmatrix}&=&\min_{(x_1,x_2)^\intercal\in\mathbb{R}^{2}} \frac{-x_1^2-x_2^2-x_1x_2+x_1}{x_2^2+1}\\
		\mbox{s.t.~~}~~ g\begin{pmatrix}
			x_1 \\ x_2
		\end{pmatrix} &=&\frac{x_1^2+2x_2^2-1}{x_2^2+1}\leq 1,\\
		&&x_1\in [-1,1],\\
		&&x_2\in \{-1,1\}.
	\end{eqnarray*}
	We can write $f\begin{pmatrix}
	x_1 \\ x_2
	\end{pmatrix}$ and $g\begin{pmatrix}
	x_1 \\ x_2
	\end{pmatrix}$ in the form 
	$$f(\mathbf{x})=\frac{\dfrac{1}{2}\mathbf{x}^\intercal A_0\mathbf{x}+\mathbf{a}_0^\intercal\mathbf{x}+c_0}{\dfrac{1}{2}\mathbf{x}^\intercal B_0\mathbf{x}+\mathbf{b}_0^\intercal\mathbf{x}+d_0}\mbox{~~and~~}g(\mathbf{x})=\frac{\dfrac{1}{2}\mathbf{x}^\intercal A_1\mathbf{x}+\mathbf{a}_1^\intercal\mathbf{x}+c_1}{\dfrac{1}{2}\mathbf{x}^\intercal B_1\mathbf{x}+\mathbf{b}_1^\intercal\mathbf{x}+d_1}$$  respectively, where $\mathbf{x}=\begin{pmatrix}
	x_1 \\ x_2
	\end{pmatrix}$, $A_0=\begin{bmatrix}
	-2 & -1\\ -1 & -2
	\end{bmatrix}$, $A_1=\begin{bmatrix}
	2 & 0\\ 0 & 4
	\end{bmatrix}$, $\mathbf{a}_0=\begin{pmatrix}
	1 \\ 0
	\end{pmatrix}$, $\mathbf{a}_1=\begin{pmatrix}
	0 \\ 0
	\end{pmatrix}$, $B_0=B_1=\begin{bmatrix}
	0 & 0\\ 0 & 2
	\end{bmatrix}$, $\mathbf{b}_0=\mathbf{b}_1=\begin{pmatrix}
	0 \\ 0
	\end{pmatrix}$, $c_0=0$, $c_1=-2$ and $d_0=d_1=1$. Let $u_1=u_2=-1$ and $v_1=v_2=1$.
	The Lagrangian of ($E_4$) is 
	\begin{equation}\label{lagE4}
	L(\mathbf{x},\lambda,\mu,\nu)=\frac{-x_1^2-x_2^2-x_1x_2+x_1}{x_2^2+1}+\lambda\left(\frac{x_1^2+x_2^2-2}{x_2^2+1}\right)+\mu(x_1^2-1)+\nu(x_2^2-1).
	\end{equation}
	Apply the KKT conditions for ($E_4$). Then,
	\begin{itemize}
		\item[(1)] $\dfrac{\partial L(\mathbf{x},\lambda,\mu,\nu)}{\partial x_1}=\dfrac{-2x_1-x_2+1}{x_2^2+1}+\dfrac{2\lambda x_1}{x_2^2+1} +2\mu x_1=0$,
		\item[(2)] $\lambda,\mu\geq0$, $\dfrac{x_1^2+x_2^2-2}{x_2^2+1}\leq 0$, $(x_1^2-4)\leq0$,
		\item[(3)] $\lambda\left(\dfrac{x_1^2+x_2^2-2}{x_2^2+1}\right)=0$,
		\item[(4)] $\mu(x_1^2-1)=0$,
		\item[(5)] $\nu=0$ and $x_2=-1$ or $x_2=1$.
	\end{itemize}
	\textbf{Case 1:} $\lambda\ne0$ or $\mu\ne0$. In this case, there are two sets of solutions $\{(x_1,x_2,\lambda,\mu,0)~\mid~x_1=\pm1, x_2=1,\lambda+2\mu=1~~\mbox{and}~~\lambda,\mu\geq0\}$ and $\{(-1,-1,\lambda,\mu,0)~\mid~\lambda+2\mu=2~~\mbox{and}~~\lambda,\mu\geq0\}$ of KKT, for the problem ($E_4$).\\\\
	\textbf{Case 2:} $\lambda=0$ and $\mu=0$. In this case, there are two solutions $(x_1,x_2,\lambda,\mu,0)=(1,-1,0,0,0)$ and $(0,1,0,0,0)$.
	
	Let $\overline{\mathbf{x}}_1=(1,1)^\intercal,\overline{\mathbf{x}}_2=(-1,1)^\intercal,\overline{\mathbf{x}}_3=(-1,-1)^\intercal,\overline{\mathbf{x}}_4=(1,-1)^\intercal$ and $\overline{\mathbf{x}}_5=(0,1)^\intercal$. Also notice that the linear independence constraint qualification is satisfied in this problem,  for $\overline{\mathbf{x}}_i\,:i=1,2,3,\cdots,5$. Direct calculations show that $\overline{\mathbf{x}}_3=(-1,-1)^\intercal$ satisfies the condition (\ref{GSC}) for all $\lambda\in[0,2]$. Other local minimizers do not satisfy the condition (\ref{GSC}). So the point $\overline{\mathbf{x}}_3=(-1,-1)^\intercal$ is the global minimizer.
\end{example}
\section{Conclusion} 
The sufficient global optimality criteria presented in this paper are useful to distinguish global minimizers and local minimizers. A number of numerical schemes are already been available to locate
local minimizers based on the KKT condition. Global optimization problems with mixed variables are inherently hard. In this context, the criteria developed in this paper to identify global minimizers
among the local minimizers would be useful for computational purposes.

\end{document}